\documentclass[reqno]{amsart}
\usepackage{amssymb,enumerate,mathrsfs,amscd}

\numberwithin{equation}{section}

\newtheorem{theorem}[equation]{Theorem}
\newtheorem{proposition}[equation]{Proposition}
\newtheorem{lemma}[equation]{Lemma}
\newtheorem{corollary}[equation]{Corollary}

\theoremstyle{definition}
\newtheorem{definition}[equation]{Definition}
\newtheorem{remark}[equation]{Remark}
\newtheorem{example}[equation]{Example}

\DeclareMathOperator{\sym}{ \sigma\!\!\!\sigma}

\def\C{\mathbb C}
\def\N{\mathbb N}
\def\R{\mathbb R}

\def\Dom{\mathcal D}
\def\K{\mathcal K}
\def\L{\mathscr L}
\def\M{\mathcal M}
\def\Oh{\mathcal O}
\def\S{\mathscr S}
\def\X{\mathcal X}
\def\Y{\mathcal Y}
\def\Z{\mathcal Z}

\def\Hor{\mathrm{Hor}}
\def\Vert{\mathrm{Vert}}

\def\bT{\,{}^{b}\hspace{-0.5pt}T}

\def\eT{\,{}^{e}\hspace{-0.5pt}T}

\def\eC{{}^eC}

\def\wg{\hspace{0.5pt}{}^{w}\hspace{-0.5pt}g}
\def\wC{\,{}^w{\!}C}
\def\wev{\,{}^w\hspace{-0.5pt}{\ev}}
\def\wT{\,{}^w\hspace{-0.5pt}T}
\def\wpi{\,{}^w\!\pi}
\def\wsym{\,{}^w\!\!\sym}

\def\embed{\hookrightarrow}
\def\eps{\varepsilon}
\def\im{i}
\def\m{\mathfrak m}

\def\open#1{\smash[t]{\overset{{}_{\,\,\circ}}{#1}{}}}

\def\set#1{\{#1\}}
\def\st{:}
\def\minus{\backslash}
\def\neutral#1{}  

\def\display#1{\mbox{\parbox{110mm} {#1}}}

\DeclareMathOperator{\ev}{ev}
\DeclareMathOperator{\Exp}{Exp}
\DeclareMathOperator{\tr}{tr}
\DeclareMathOperator{\Tr}{Tr}
\DeclareMathOperator{\Diff}{Diff}
\DeclareMathOperator{\spec}{spec}
\DeclareMathOperator{\Hom}{Hom}
\DeclareMathOperator{\opm}{op_M}
\DeclareMathOperator{\Op}{Op}
\DeclareMathOperator{\ord}{ord}

\begin{document}
\title{On the closure of elliptic wedge operators}
\thanks{Work partially supported by the National Science Foundation, Grants DMS-0901173 and DMS-0901202}

\author{Juan B. Gil}
\address{Penn State Altoona\\ 3000 Ivyside Park \\ Altoona, PA 16601-3760}
\email{jgil@psu.edu}
\author{Thomas Krainer}
\address{Penn State Altoona\\ 3000 Ivyside Park \\ Altoona, PA 16601-3760}
\email{krainer@psu.edu}
\author{Gerardo A. Mendoza}
\address{Department of Mathematics\\ Temple University\\ Philadelphia, PA 19122}
\email{gmendoza@temple.edu}

\begin{abstract}
We prove a semi-Fredholm theorem for the minimal extension of elliptic operators on manifolds with wedge singularities and give, under suitable assumptions, a full asymptotic expansion of the trace of the resolvent.
\end{abstract}

\subjclass[2010]{Primary: 58J50; Secondary:  35P05, 58J32, 58J05}
\keywords{Manifolds with edge singularities, elliptic operators, resolvents, trace asymptotics}

\maketitle

\section{Introduction}

Differential wedge operators arise, for example, when writing a regular differential operator in cylindrical coordinates in a tubular neighborhood of an embedded closed submanifold $\Y$ in some ambient closed manifold $\X$. Differential-geometrically, this blow-up procedure gives rise to a compact manifold $\M$ with boundary, whose boundary $\partial\M$ is the total space of a sphere bundle over $\Y$. The blow-down map $\M \to \X$ takes $\partial\M$ to $\Y$ and collapses each fiber to a point. The pull-back of any Riemannian metric on $\X$ to $\M$ is an example of an incomplete edge metric on $\M$.

Generalizing this, manifolds with wedge singularities are represented by compact manifolds $\M$ with boundary, where the boundary is the total space of a locally trivial fiber bundle $\wp : \partial\M \to \Y$ with typical fiber $\Z$:
\begin{equation}\label{BdyFiberBundle}
\display{\begin{center}
\begin{picture}(120,52)
\put(45,42){$\Z\embed\partial\M$	
\put(-13,-20){$\bigg\downarrow$}
\put(-5,-18){\scriptsize $\wp$}
\put(-14,-40){$\Y$}
}
\end{picture}
\end{center}}
\end{equation}
The manifolds $\Z$ and $\Y$ are assumed to be closed. In addition, $\Y$ is assumed to be connected for notational simplicity. The actual singular space (whose singularity is the edge, $\Y$) is $\M/{\sim}$, where $\sim$ collapses the fibers to points. Wedge differential operators represent a particular class of operators on such compact manifolds with fibered boundaries.

In the present paper we study the closure $A_{\min}$ of an elliptic differential wedge operator
\begin{equation*}
A : C_c^{\infty}(\open\M;E) \subset x^{-\gamma}L^2_b(\M;E) \to x^{-\gamma}L^2_b(\M;F).
\end{equation*}
Here $x$ is a defining function for the boundary, and $\gamma \in \R$ is a fixed weight. 

Under natural ellipticity assumptions, we prove that $A_{\min}$ is semi-Fredholm with finite-dimensional kernel and closed range, and we identify the minimal domain as the space $x^{-\gamma+m}H^m_{e}(\M;E)$; see \cite{Maz91}. Under only mild natural assumptions, we construct a left inverse $B(\lambda) : x^{-\gamma}L^2_b(\M;E) \to \Dom_{\min}(A)$ of $A_{\min} - \lambda$ for large $|\lambda|$ in a given sector $\Lambda \subset \C$. We then show that the trace of $B(\lambda)$ admits a full asymptotic expansion as $|\lambda| \to \infty$, $\lambda \in \Lambda$. 

Under strong additional conditions we show the existence of the resolvent family $\big(A_{\min}-\lambda\big)^{-1}$ on $\Lambda$, show that $\Lambda$ is a sector of minimal growth, and prove that for $\ell > \dim\M/{\ord}(A)$ and arbitrary
$\varphi \in C^{\infty}(\M;\Hom(E))$,
\begin{equation}\label{IntroResolventTrace}
\Tr\Bigl(\varphi\bigl(A_{\min}-\lambda\bigr)^{-\ell}\Bigr) \sim \sum_{j=0}^{\infty}\sum_{k=0}^{m_j}
\alpha_{jk}\lambda^{\frac{\dim\M-j}{\ord(A)}-\ell}\log^k(\lambda) \text{ as } |\lambda| \to \infty.
\end{equation}
Here $m_j \leq 1$ for all $j$ and $m_j = 0$ for $j \leq \dim \Z$.

By standard methods, the expansion \eqref{IntroResolventTrace} leads to results about the short time asymptotics of the heat trace when $A_{\min}$ is sectorial, and to results about the meromorphic structure of the $\zeta$-function when $A_{\min}$ is positive.

While there is a fair amount of literature regarding expansions of the form \eqref{IntroResolventTrace} in the cases of boundary value problems ($\dim \Z=0$) and of elliptic cone operators ($\dim\Y=0$), little is known for general wedge operators. For the special case of the Friedrichs extension of a semibounded second order wedge operator, a major contribution is due to Br\"uning and Seeley \cite{BruSee91}. Mazzeo and Vertman have an alternative proof\footnote{Work in progress, personal communication.} of Br\"uning's and Seeley's results using methods based on Mazzeo's edge calculus; see \cite{Maz91}.  

The present paper is a first step towards a systematic study of realizations of general elliptic wedge operators and their spectral theory. In work in progress, our goal is to study both well-posedness and trace asymptotics of elliptic wedge problems. For the case where $\dim\Y=0$, the most general result concerning trace asymptotics is proved in \cite{GKM5}.

\medskip
The structure of the paper is as follows. In Section~\ref{Sec-WedgeOps} we review the class of differential wedge operators and their symbols. Section~\ref{Spaces} collects a number of definitions and constructions that pertain to the function spaces playing a role in the analysis of wedge operators. Section~\ref{MinDomain} contains a theorem on the structure of the minimal domain and the semi-Fredholm result for the minimal extension (Theorem~\ref{DminTheorem}). Section~\ref{Sec-ResolventFamily} is devoted to the investigation of the resolvent family of the minimal extension and the asymptotic expansion of its trace (Theorem~\ref{MinLeftInverseandExpansion} and Corollary~\ref{Resolventexpansion}). The proofs of our main results in Sections~\ref{MinDomain} and \ref{Sec-ResolventFamily} rely heavily on a suitable parametrix constructed in Section~\ref{ParametrixConstruction}. Here we make substantial use of methods developed by Schulze on pseudodifferential edge operators; see e.g. \cite{SchuNH}. Alternatively, one could develop modifications of Mazzeo's approach to edge operators by incorporating parameters to accommodate the situation at hand. Finally, Section~\ref{ExpansionProof} contains the proof of the asymptotic expansion of the resolvent trace (Theorem~\ref{MinLeftInverseandExpansion}).

\section{Differential wedge operators and their symbols}\label{Sec-WedgeOps}

\subsection*{Differential wedge operators}

Let $\M$ be a $C^{\infty}$-smooth compact $d$-dimen\-sio\-nal manifold with boundary $\partial\M$, $\wp : \partial\M \to \Y$ be a smooth locally trivial fiber bundle with typical fiber $\Z$. As mentioned in the introduction, $\Y$ and $\Z$ are closed manifolds and $\Y$ is connected. We will write $q$ for the dimension of $\dim \Y$, $n$ for that of $\Z$, and $\Z_y$ for the fiber of $\partial\M$ over $y\in \Y$. Let $E, F\to\M$ be smooth vector bundles.

Recall from \cite{Maz91} that $\eT\M\to\M$ is the vector bundle whose space of sections is the space $\eC^\infty(\M;T\M)$ of smooth vector fields on $\M$ which over the boundary are tangent to the fibers of $\wp$; this is a Lie algebra. Recall also that if $E$, $F$ are vector bundles over $\M$, then $\Diff_e(\M;E,F)$, the space of edge operators, is the space of smooth linear differential operators $C^\infty(\M;E)\to C^\infty(\M;F)$ acting locally by matrices of elements of the enveloping algebra of $\eC^\infty(\M;\C T\M)$ and $C^\infty(\M)$. The subspace of operators of order $m$ is $\Diff^m_e(\M;E,F)$, or $\Diff^m_e(\M;E)$ if $F=E$, or simply $\Diff^m_e(\M)$ in the case of scalar operators.

The space of differential {\em wedge operators} of order $m$ is 
\begin{equation*}
x^{-m}\Diff^m_e(\M;E,F).
\end{equation*}
Here and throughout the paper $x : \M \to \R$ is any smooth defining function for $\partial\M$, i.e. $x \geq 0$ on $\M$ with $x = 0$ precisely on $\partial\M$, and $dx \neq 0$ on $\partial\M$. Note that $x^{-m}\Diff^m_e(\M;E,F)$ does not depend on the choice of defining function.

Let $p_0\in \partial\M$ and $\wp(p_0) = y_0$. Let $z_1,\dotsc,z_n$ be functions defined near $p_0$ in $\M$ whose restriction to $\Z_{y_0}$ give coordinates near $p_0$. Let $y_1,\dotsc,y_q$ be coordinates for $\Y$ near $p_0$. Write also $y_j$ for an extension of $\wp^*y_j$ to a neighborhood of $\Z_{y_0}$ in $\M$. The functions $x,y,z$ are coordinates for $\M$ near $p_0$; such coordinates will be called adapted. 

Locally, near a point $p_0 \in \partial\M$, a differential wedge operator $A \in x^{-m}\Diff^m_e(\M)$ can be represented in terms of adapted coordinates $x,y,z$ near $p_0$ as
\begin{equation}\label{LocalEdgeOperator1}
x^{-m}\!\!\!\sum_{k+|\alpha|+|\beta|\leq m} \!\! a_{k,\alpha,\beta}(x,y,z) (xD_x)^k(xD_y)^{\alpha}D_z^{\beta},
\end{equation}
alternatively as
\begin{equation}\label{LocalEdgeOperator2}
x^{-m}\!\!\!\sum_{k+|\alpha|+|\beta|\leq m} \!\! b_{k,\alpha,\beta}(x,y,z) x^kD_x^k(xD_y)^{\alpha}D_z^{\beta},
\end{equation}
with coefficients $a_{k,\alpha,\beta}$ or $b_{k,\alpha,\beta}$ which are smooth up to $x = 0$. Here we use standard multi-index notation and the usual convention $D = -\im \partial$ for the differential, and we write $(xD_y)^{\alpha}$ for $x^{|\alpha|}D_y^{\alpha}$. 

\subsection*{The wedge-cotangent bundle}\label{wcotangentbundle}

Let $\wC^\infty(\M;T^*\M)$ be the space of smooth one-forms on $\M$ whose pullback to each fiber of $\wp$ vanishes. Since $\wC^\infty(\M;T^*\M)$ is a locally free finitely generated $C^\infty(\M)$-module, there is a vector bundle
\begin{equation*}
\wpi:\wT^*\M\to\M
\end{equation*}
whose space of smooth sections is isomorphic to $\wC^\infty(\M;T^*\M)$ (see Swan \cite{Swan62}); this is the wedge-cotangent bundle of $\M$, the $w$-cotangent for short. Since moreover $\wC^\infty(\M;T^*\M)$ is a submodule of $C^\infty(\M;T^*\M)$, there is a vector bundle homomorphism
\begin{equation}\label{wev}
\wev^*:\wT^*\M\to T^*\M
\end{equation}
such that the induced map 
\begin{equation*}
\iota:C^\infty(\M;\wT^*\M)\to \wC^\infty(\M;T^*\M)
\end{equation*}
is an isomorphism of $C^\infty(\M)$-modules. The homomorphism $\wev^*$ is an isomorphism over $\open \M$, the interior of $\M$. 

Again, let $p_0\in \partial\M$ and pick adapted coordinates $x,y,z$ in a neighborhood $U$ of $p_0$. A smooth $1$-form $\alpha=a dx+\sum_j b_j dy_j+\sum_k c_k dz_k$ has the property that its pullback, $\sum_k c_k dz_k$, to each set $U\cap \Z_y$ vanishes if and only if the coefficients $c_j$ vanish when $x=0$. This easily gives that the differentials
\begin{equation}\label{FrameWTStar}
dx,\ dy_j,\ xdz_k, \;\; j=1,\dotsc,q, \ \;  k=1,\dotsc,n,
\end{equation}
furnish a local frame for $\wT^*\M$ near $p_0$. More properly, these differentials are the image by $\iota$ of a local frame of $\wT^*\M$.

\subsection*{The $w$-principal symbol}

Every $A \in x^{-m}\Diff^m_e(\M;E,F)$ has an invariantly defined principal symbol $\wsym(A)$; a section of the homomorphism bundle
\begin{equation*}
\Hom(\wpi^*E,\wpi^*F)\to  \wT^*\M
\end{equation*}
which is homogeneous of degree $m$ in the fibers. This is the {\em wedge symbol} of $A$.

Namely, if $A\in x^{-m}\Diff_{e}(\M;E,F)$, then its standard principal symbol (over $\open\M$),
\begin{equation*}
\sym(A)\in C^\infty(T^*\open\M\minus 0;\Hom(\pi^*E,\pi^*F)),
\end{equation*}
gives a section
\begin{equation}\label{eprincvsprincsymb}
\wsym(A) = \sym(A)\circ\wev^*
\end{equation}
of $\Hom(\wpi^*E,\wpi^*F)$ over the part of $\wT^*\M\minus 0$ over $\open\M$ which extends smoothly as a section
\begin{equation*}
\wsym(A)\in C^\infty(\wT^*\M\minus 0;\Hom(\wpi^*E,\wpi^*F)).
\end{equation*}
To see this, let $p_0\in \partial\M$, let $x,y,z$ be adapted coordinates in the neighborhood $U$ of  $p_0$, and write $A$ in the form \eqref{LocalEdgeOperator1}. If $\pmb\nu=\xi dx+ \sum \eta_j dy_j+ \sum_k\zeta_k xdz_k$ is a covector over a point $p\in \open\M\cap U$ of coordinates $x,y,z$, then 
\begin{equation*}
\sym(A)(\pmb\nu)=x^{-m}\!\!\!\sum_{\ell+|\alpha|+|\beta|=m} \! a_{\ell,\alpha,\beta} (x\xi)^\ell (x\eta)^\alpha (x\zeta)^\beta,
\end{equation*}
an expression that clearly extends smoothly down to $x=0$. Of course, using the representation \eqref{LocalEdgeOperator2} instead of \eqref{LocalEdgeOperator1} gives the same result. 

The wedge symbol can also be obtained through an oscillatory test using as phase a function $f \in C^{\infty}(\M)$ such that $f|_{\partial\M} = \wp^*g$ with $g \in C^{\infty}(\Y)$. Indeed, for such a function one has that $df \in \wC^{\infty}(\M;T^*\M)$, and hence $df=\wev^*\phi$ for a unique section $\phi$ of $\wT^*\M$. Then, if $\psi$ is a section of $E$ on $\M$, the standard formula
\begin{equation*}
\sym(A)(df)=\lim_{\varrho\to\infty} \varrho^{-m}e^{-i\varrho f(p)}Ae^{i\varrho f}\psi
\end{equation*}
over $\open\M$ gives
\begin{equation*}
\wsym(A)(\phi)=\lim_{\varrho\to\infty} \varrho^{-m}e^{-i\varrho f(p)}Ae^{i\varrho f}\psi,
\end{equation*}
also over $\open \M$. The latter extends smoothly to $\partial\M$.

\begin{definition}\label{eellipticity}
An operator $A \in x^{-m}\Diff^m_e(\M;E,F)$ is called {\em $w$-elliptic} if its $w$-principal symbol $\wsym(A)$ is invertible on $\wT^*\M\setminus 0$.
\end{definition}

In view of \eqref{eprincvsprincsymb}, a $w$-elliptic operator is in particular elliptic on $\open\M$.

\subsection*{The normal family}

Let $x : \M \to \R$ be a defining function for $\partial\M$, positive in the interior. Since $dx:T_{\partial\M}\M\to \R$ has kernel $T\partial\M$, it induces a function 
$x_{\wedge} : N \partial\M \to \R$ on the normal bundle of $\partial\M$ in $\M$. Let
\begin{equation*}
N_+ \partial\M = \set{v\in N \partial\M: x_{\wedge}v \geq 0}
\end{equation*}
be the closed inward pointing half of the normal bundle. We view it as a manifold with boundary; the latter, the zero section, is  identified with $\partial\M$. The function $x_\wedge$ is a defining function for the boundary which also gives a trivialization of $N \partial\M$ by the class of a vector field $X$ of $\M$ along $\partial\M$ such that $Xx=1$: $N_+ \partial\M \cong \partial\M \times 
\overline{\R}_+$. We will write $\partial_{x_\wedge}$ for the vertical vector field such that $\partial_{x_\wedge}x_\wedge=1$. This vector field depends on $x_\wedge$, but $x_\wedge\partial_{x_\wedge}$ does not. Later we will drop the subscript $\wedge$ from this notation.

Let
\begin{equation}\label{defWPwedge}
\pi_\wedge:N_+\partial\M\to \partial\M
\end{equation}
be the canonical projection. The map
\begin{equation}\label{ModelBoundaryFibration}
\wp_\wedge=\wp\circ \pi_\wedge:N_+\partial\M\to\Y
\end{equation}
makes $N_+\partial\M$ into the total space of a locally trivial bundle with typical fiber $\Z^{\wedge} = \overline{\R}_+ \times \Z$.

For descriptive purposes it is convenient to pick a tubular neighborhood map with some care. Pick a vector field $X$ on $\M$ such that $Xx = 1$ near $\partial\M$. It determines a trivialization of $N_+\partial\M$ and a flow on $\M$ which together give a tubular neighborhood map
\begin{equation*}
N_+\partial\M\supset V \xrightarrow{\varphi} U\subset \M,
\end{equation*}
where
\begin{equation*}
V= \set{v \in N_+\partial\M : 0 \leq x_{\wedge}v < \eps}.
\end{equation*}
Thus if $v\in V$, then $[0,1]\ni t\mapsto \varphi(tv)\in \M$ is the integral curve of $X$ starting at $p=\pi_\wedge(v)$. The map $\varphi$ enables us to pass from $N_+\partial\M$ to the manifold $\M$ and vice versa near the boundary. The fact that  $x\circ \varphi = x_\wedge$ allows us to drop the $\wedge$ from the notation $x_\wedge$ in most cases.

Let $E \to \M$ be a Hermitian vector bundle, fix a Hermitian connection and let $E^{\wedge} = \pi_\wedge^*E$ have the induced metric and connection. We get a vector bundle morphism\begin{equation*}
\Phi:E^\wedge_V\to E_U
\end{equation*}
covering $\varphi$ by sending the point $E^\wedge_v$ obtained from parallel transport of $\eta\in E^\wedge_p$ along $t\mapsto tv$ ($v\in V\subset N_+\partial\M$, $p=\pi_\wedge(v)\in \partial\M\ (=\partial N_+\partial\M)$) to the point of $E_{\Phi(v)}$ which results from parallel transport of $\eta$ along $t\mapsto \varphi(tv)$. Clearly, the morphism $\Phi$ is an isometry. We generally use the notation $E_U$ to mean the part of $E$ over $U$, likewise $E^\wedge_V$. 

If $F\to \M$ is another Hermitian vector bundle with Hermitian connection, the similarly constructed map will also be a bundle isomorphism $F^\wedge_V\to F_U$; we will denote it again by $\Phi$. 

With every operator $A \in x^{-m}\Diff^m_e(\M;E,F)$ we can associate the operator
\begin{equation*}
\Phi^* A \Phi_* : C_c^{\infty}(\open V;E^{\wedge}) \to
C^{\infty}(\open V;F^{\wedge})
\end{equation*}
where of course $\open V= V\cap \open N_+\partial\M$. (The meaning of the notation $\Phi_*$ and $\Phi^*$ should be clear; we also use this notation rather than $\varphi_*$ or $\varphi^*$ when either of the bundles is trivial.) It is easy to see that 
\begin{equation*}
\Phi^* A \Phi_* \in x^{-m}\Diff^m_e(V;E^\wedge,F^\wedge).
\end{equation*}

Let $\tau_\varrho:E^\wedge\to E^\wedge$, $\rho>0$, denote the vector bundle morphism that sends $E^\wedge_v$ to $E^\wedge_{\varrho v}$ using parallel transport along $t\mapsto tv$; $\tau_\varrho$ is really independent of the original connection on $E$. For $u \in C^{\infty}(\open N_+\partial\M;E^{\wedge})$ define $\kappa_{\varrho}u$ by
\begin{equation}\label{kappadef}
(\kappa_{\varrho}u)(v) = \varrho^{\gamma}\tau_{\varrho}^{-1}(u(\varrho v)),\quad v\in \open N_+\partial\M;
\end{equation}
the factor $\varrho^\gamma$ is for the moment arbitrary but will eventually result in $\kappa_\varrho$ extending to an isometry of a suitable $L^2$ space; see \eqref{Aunbounded}. Clearly 
\begin{equation*}
\R_+\ni \varrho\mapsto \kappa_{\varrho} : C_c^{\infty}(\open N_+\partial\M;E^{\wedge}) \to
C_c^{\infty}(\open N_+\partial\M;E^{\wedge})
\end{equation*}
is a group homomorphism, $\R_+$ with the multiplicative structure and the family of linear invertible maps on $C_c^{\infty}(\open N_+\partial\M;E^{\wedge})$ with composition. Clearly, $\kappa_\varrho$ also defines a map on sections over individual fibers of $N_+\partial\M\to\partial\M$.

If $u \in C_c^{\infty}(\open N_+\partial\M;E^{\wedge})$, then the support of $\kappa_{\varrho}u$ will be contained in $\open V$ once 
$\varrho > 0$ is sufficiently large, hence $\Phi_* \kappa_\varrho u$ will be defined as an element of $C_c^\infty(\open U;E)$. 

\begin{proposition}\label{WedgeSymbolWellDefined}
Let $A\in x^{-m}\Diff^m_e(\M;E,F)$, let $g\in C^\infty(\Y)$ be a real-valued function, and let $u\in C_c^\infty(\open N_+\partial\M,E^\wedge)$. Then the limit 
\begin{equation}\label{NormalFamilyAsLimit}
\lim_{\varrho\to \infty} \varrho^{-m}\kappa^{-1}_{\varrho}e^{-\im \varrho \wp_\wedge^*g} \Phi^* A \Phi_* e^{\im \varrho \wp_\wedge^*g} \kappa_{\varrho}u =A_\wedge (dg)u
\end{equation}
exists. As the notation indicates, the resulting operator depends only on $dg$. Furthermore, it commutes with multiplication by any function of the form $\wp_\wedge^*h$, $h\in C^\infty(\Y)$. Therefore it defines, for each $y\in \Y$ and $\pmb \eta\in T^*_y\Y$, a differential operator $A_\wedge(\pmb\eta)$ on $\Z^\wedge_y=\wp_\wedge^{-1}(y)$. This operator belongs to $x^{-m}\Diff_b^m(\Z^\wedge_y;E^\wedge_{\Z^\wedge_y},F^\wedge_{\Z^\wedge_y})$, and if $A$ is $w$-elliptic, then $A_\wedge(\pmb\eta)$ is $b$-elliptic.
\end{proposition}

\begin{proof}
Let $y_0\in \Y$ and $p_0\in \wp^{-1}(y_0)$. We trivialize the vector bundles $E$ and $F$ in a neighborhood of $p_0$ in $\partial\M$ and extend the trivializations using frames that are parallel with respect to the integral curves of $X$, and do likewise with $E^\wedge$ and $F^\wedge$ using $\partial_{x_\wedge}$. This both reduces the problem to dealing with scalar operators and allows us to identify $U$ with $V$. Using coordinates $(x,y,z)$ on $N_+\partial\M$ where $x$ is just $x_\wedge$, the $y_j$ are lifted from $\Y$, and the $z_k$ are lifted from an extension of coordinates of $\Z_y$ near $p_0$, we have $(\kappa_\varrho u)(x,y,z)=\varrho^\gamma u(\varrho x,y,z)$, and
\begin{equation*}
\begin{gathered}
xD_x (e^{\im\varrho\wp_\wedge^*g}\kappa_\varrho u) = e^{\im\varrho\wp_\wedge^*g}\kappa_\varrho (xD_x u), \quad
D_{z_k}(e^{\im\varrho\wp_\wedge^*g}\kappa_\varrho u)=e^{\im\varrho\wp_\wedge^*g}\kappa_\varrho (D_{z_k}u),\\
 xD_{y_j}(e^{\im\varrho\wp_\wedge^*g}\kappa_\varrho u)=e^{\im\varrho\wp_\wedge^*g}\kappa_\varrho \big ( (x\wp_\wedge^* g_{y_j} + \tfrac{x}{\varrho} D_{y_j}) u\big).
\end{gathered}
\end{equation*}
Writing $A$ in the form \eqref{LocalEdgeOperator1} we thus have
\begin{align*}
\varrho^{-m}&\kappa^{-1}_{\varrho}e^{-\im \varrho \wp_\wedge^*g} \Phi^* A \Phi_* e^{\im \varrho \wp_\wedge^*g} \kappa_{\varrho}u\\
&= x^{-m} \hspace{-1em} \sum_{k+|\alpha|+|\beta|\leq m}a_{k\alpha\beta} a(x/\varrho,y,z) (xD_x)^k\big( (x\wp_\wedge^*g_y)^\alpha+\Oh(\varrho^{-1})\big) D_z^\beta u,
\end{align*}
where $\wp_\wedge^*g_y$ means the coordinate gradient of $\wp_\wedge^*g$. Consequently, the limit as $\varrho\to\infty$ exists, as claimed, and depends only pointwise on $\wp_\wedge^*dg$. Thus, if $dg$ is the covector $\pmb\eta=\sum \eta_j dy_j$ at $y$, then

\begin{equation}\label{WedgeSymbolLocal}
A_\wedge(\pmb \eta)= x^{-m} \!\!\! \sum_{k+|\alpha|+|\beta|\leq m} \!\! a_{k,\alpha,\beta}(0,y,z)(xD_x)^k (x\eta)^\alpha D_z^\beta
\end{equation}
from which the rest of the statements in the proposition follow.
\end{proof}

The only arbitrary element in \eqref{NormalFamilyAsLimit} is the map $\Phi$. We will leave it to the reader to verify that the limit is actually independent of the tubular neighborhood map. The interested reader may also verify that if $A$ is symmetric, then again $A_\wedge(\pmb \eta)$ is symmetric for each $\pmb \eta$.

\begin{definition}\label{NormalFamily}
The {\em normal family} of $A$ is the family
\begin{equation*}
T^*\Y\ni \pmb \eta \mapsto A_\wedge(\pmb\eta) \in x^{-m}\Diff_b^m(\Z^\wedge_y;E_{\Z^\wedge_y}^\wedge,F_{\Z^\wedge_y}^\wedge), \quad y=\pi_{\Y}\pmb \eta.
\end{equation*}
\end{definition}

It follows immediately from \eqref{NormalFamilyAsLimit} (also from \eqref{WedgeSymbolLocal} if one prefers working with coordinates) that the normal family is $\kappa$-homogeneous of degree $m$, i.e.,
\begin{equation}\label{WedgeSymbolKappaHom}
A_{\wedge}(\varrho\, \pmb\eta) = \varrho^m \kappa_{\varrho} A_{\wedge}(\pmb\eta)\kappa_{\varrho}^{-1}.
\end{equation}

\subsection*{The conormal symbol}

Again suppose $A \in x^{-m}\Diff^m_e(\M;E,F)$. Then $P = x^mA$ belongs to $\Diff^m_e(\M;E,F)$, and since $\Diff^m_e(\M;E,F) \subset \Diff_b^m(\M;E,F)$, the operator $P$ has an indicial family
\begin{equation*}
\widehat{P}(\sigma) : C^{\infty}(\partial\M;E_{\partial\M}) \to C^{\infty}(\partial\M;F_{\partial\M}) \; \text{ for } \sigma \in \C,
\end{equation*}
which we recall, is given as follows. If $v \in C^{\infty}(\partial\M;E)$ and $u \in C^{\infty}(\M;E)$ is any extension of $v$, then
\begin{equation*}
\widehat{P}(\sigma)v = x^{-i\sigma}Px^{i\sigma}u\big|_{\partial\M}.
\end{equation*}
This defines a holomorphic family 
\begin{equation*}
\C\ni \sigma\mapsto \widehat{P}(\sigma) \in \Diff^m(\partial\M;E_{\partial\M},F_{\partial\M}).
\end{equation*}
Likewise, if $g\in C^\infty(\Y)$ is real-valued, then 
\begin{equation*}
P_\wedge(dg)=x_\wedge^m A_\wedge(dg) \in \Diff^m_e(N_+\partial\M;E^\wedge,F^\wedge)
\end{equation*}
has an indicial family, which happens to coincide with that of $P$ using the canonical identification of $\partial\M$ with the zero section of $N_+\partial\M$. Hence the indicial family of $P_\wedge$ depends only on the base point of $T^*\Y$. And since $A_\wedge$ commutes with multiplication by any function of the form $\wp_\wedge^*h$, $h\in C^\infty(\Y)$, so does $P_\wedge$, hence also $\widehat P$. As a consequence, each of the operators $\widehat P(\sigma)$ determines yet another family of operators,
\begin{equation*}
\Y\times \C\ni (y,\sigma)\mapsto \widehat P(y,\sigma):C^\infty(\Z_y;E_{\Z_y})\to C^\infty(\Z_y;F_{\Z_y})
\end{equation*}
on each of the fibers $\Z_y=\wp^{-1}(y)$.

\begin{definition}\label{ConSymb}
The {\em conormal symbol} of $A\in x^{-m}\Diff^m_e(\M;E,F)$ is the family
\begin{equation*}
T^*\Y\times \C \ni (\pmb\eta,\sigma)\mapsto \widehat A(\pmb\eta,\sigma)
\equiv \widehat P(y,\sigma)\in \Diff^m(\Z_y;E_{\Z_y},F_{\Z_y}),\;\; y=\pi_\Y\pmb \eta.
\end{equation*}
\end{definition}

Since $\widehat A(\pmb \eta,\sigma)$ really depends only on $y=\pi_\Y\pmb \eta$ (and $\sigma$), we will often write $\widehat A(y,\sigma)$ instead. In the coordinates in which \eqref{WedgeSymbolLocal} holds we have
\begin{equation*}
\widehat A(y,\sigma)= \!\!\sum_{k+|\alpha|+|\beta|\leq m}a_{k,\alpha,\beta}(0,y,z)\sigma^k D_z^\beta.
\end{equation*}
The $e$-spectrum of $A$ is the set
\begin{equation*}
\spec_{e}(A) = \set{(y,\sigma) : \widehat{A}(y,\sigma) \text{ is not invertible}}
\end{equation*}
and the $b$-spectrum at $y\in\Y$ is
\begin{equation*}
\spec_{b,y}(A) = \set{\sigma \in \C : \widehat{A}(y,\sigma) \text{ is not invertible}}.
\end{equation*}
We will also make use of the set
\begin{equation*}
\pi_\C \spec_{e}(A) = \bigcup_{y\in\Y} \spec_{b,y}(A),
\end{equation*}
where $\pi_\C:\Y\times\C\to\C$ is the canonical projection onto the second factor. 

\smallskip
If $A\in x^{-m}\Diff^m_e(\M;E,F)$ is $w$-elliptic, then $A_\wedge(\pmb \eta)$ is $b$-elliptic for each $\pmb \eta\in T^*\Y$. Hence by the elliptic theory of $b$-operators, the operators $\widehat A(y,\sigma)$ are elliptic and we have:

\begin{lemma}\label{ConormalSymbolFredholm}
Let $A\in x^{-m}\Diff_e^m(\M;E,F)$ be $w$-elliptic and let $y\in \Y$ be fixed. The operator family
\begin{equation*}
\sigma \mapsto \widehat{A}(y,\sigma) : C^{\infty}(\Z_y;E_{\Z_y}) \to 
C^{\infty}(\Z_y;F_{\Z_y})
\end{equation*}
is a holomorphic Fredholm family with a finitely meromorphic inverse. Moreover, each set $\spec_{b,y}(A)$ is discrete, and for every $M>0$,
\begin{equation*}
\spec_{b,y}(A)\cap \set{\sigma\in \C: |\Im\sigma|\leq M}
\end{equation*}
 is finite.
\end{lemma}

More precisely, for each $y\in \Y$, 
\begin{equation*}
\C \ni \sigma \mapsto \widehat{A}(y,\sigma) \in \L\bigl(H^s(\Z_y;E_{\Z_y}),H^{s-m}(\Z_y;F_{\Z_y})\bigr)
\end{equation*}
is, for every $s$, a holomorphic family of elliptic differential operators with finitely meromorphic inverse. Standard elliptic regularity gives that the set $\spec_{b,y}(A)$ does not depend on the value of $s \in \R$.

While the normal family is independent of choices, the indicial family does depend on the defining function, however in a very mild way: changing the defining function to $\tilde x$ only results in a conjugation by $(\tilde{x}/x)^{i\sigma}$.

\subsection*{$w$-Laplacians} 
A $w$-metric $\wg$ is a smooth Riemannian metric on the dual $\wT\M$ of the $w$-cotangent bundle. Since \eqref{wev} is an isomorphism over $\open \M$, its dual map
\begin{equation*}
\wev:T\M\to \wT\M
\end{equation*}
gives a natural isomorphism between $T\open \M$ and the part of $\wT\M$ over $\open\M$. By way of this isomorphism, we get a standard Riemannian metric on $\open \M$ which we also denote by $\wg$. Near the boundary we can write $\wg$ as a $2$-cotensor in the forms \eqref{FrameWTStar} whose coefficients are smooth up to the boundary and form a symmetric matrix which is positive definite up to the boundary.

A $w$-Laplacian is a Laplace-Beltrami operator with respect to a $w$-metric. Such operators are elliptic wedge operators which provide an important class of examples in the theory.

Note that the Riemannian density induced by $\wg$ is of the form $x^n\m$, where $\m$ is a smooth positive density on $\M$ up to the boundary.

\section{Spaces}\label{Spaces}

Fix a $b$-density $\m_b$ on $\open\M$, i.e. a density on $\open\M$ of the form $\m_b=x^{-1}\m$ for some smooth positive density $\m$ on $\M$. The $L^2$ space of sections of the Hermitian vector bundle $E\to\M$ defined using these data is denoted $L^2_b(\M;E)$, as usual. In the context of edge operators, one of the fundamental spaces is:

\begin{definition}[Mazzeo \cite{Maz91}]\label{Hedge}
For $s \in \N_0$ let
\begin{equation*}
H^s_{e}(\M;E) = \{u \in {\mathcal D}'(\M;E) : Bu \in L^2_b(\M;E) \text{ for all } B \in \Diff_{e}^k(\M;E), \; k \leq s\}.
\end{equation*}
Define $H^s_{e}(\M;E)$ for general $s\in\R$ using interpolation and duality.
\end{definition}

The $b$-density $\m_b$ may be viewed as a density on $\bT\M$, smooth up to the boundary. Using the canonical section $x\partial_x$ of $\bT\M$ along $\partial\M$ we obtain a smooth density $\m_{\partial\M}$ on $\partial\M$ as follows: given a basis  $\mathbf v=(v_1,\dotsc,v_n)$ of $T_p\partial\M$, let $\m_{\partial\M}$ at $\mathbf v$ be the number $\m_b(v_1,\dotsc,v_n,x\partial_x)$. With this density we define $\m_{N_+}=\m_{\partial\M}\otimes|x_\wedge^{-1} dx_\wedge|$ on $N_+\partial\M$. This resulting density is independent of the choice of defining function $x$.

To obtain densities on the various $\Z_y$ and on $\Z^\wedge_y$, suppose that a smooth density has been given  on $\Y$. For example, let $\m_\Y=\wp_*\m_{\partial\M}$. Then there is a canonical density $\m_{\Z_y}$ induced  on each fiber $\Z_y$ of $\wp$, with which in turn we define $\m_{\Z^\wedge_y}=\m_{\Z_y}\otimes |x_\wedge^{-1} dx_\wedge|$ on $\Z^\wedge_y$. The latter is again independent of the choice of $x$.

Thus we have $L^2$ spaces of sections of the various versions of $E$ on all the manifolds involved in the previous two paragraphs. 

\medskip
We will also need to compare domains for $A_\wedge(\pmb\eta)$ as $\pmb \eta$ varies over $T^*\Y$. We do this by fixing the $L^2$ spaces and transferring all dependence in $y$ of measures and Hermitian metrics to the operators. This entails as a first step making
\begin{equation*}
\mathfrak p:\bigsqcup_{y\in\Y}L^2(\Z_y;E_{\Z_y})\to \Y
\end{equation*}
into a (Hilbert space) vector bundle, as described in the next paragraphs. The definition of $\mathfrak p$ is clear.

First, fix a Riemannian metric $g$ on $\partial\M$. Let $\Vert\subset T\partial\M$ be the kernel of $d\wp:T\partial\M\to T\Y$ and let $\Hor=\Vert^\perp \subset T\partial\M$ be the orthogonal subbundle. Assume $g$ has the property that
\begin{equation}\label{Uniformity}
\forall p\in \partial\M,\ \forall u,v\in \Hor_p:\quad \wp_*u=\wp_*v \implies g(u,u)=g(v,v).
\end{equation}
Thus $g$ determines a metric $g_\Y$ on $\Y$. Let $\gamma:[a,b]\to \Y$ be a continuous piecewise smooth curve. Recall that for each $p\in \Z_{\gamma(a)}$ there is a unique horizontal lift of $\gamma$, a curve
\begin{equation*}
[a,b]\ni t\mapsto \hat\gamma_t(p)\in\partial\M
\end{equation*}
with derivative in $\Hor$ such that $\wp(\hat \gamma_t(p))=\gamma(t)$ and $\hat \gamma_a(p)=p$. The maps
\begin{equation*}
\Z_{\gamma(a)}\ni p \mapsto \hat \gamma_t(p) \in \Z_{\gamma(t)} 
\end{equation*}
are diffeomorphisms for each $t$. Let $J_{\gamma(t)}\in C^\infty(\Z_{\gamma(t)})$ be the function such that $\hat\gamma_{t,*}\m_{\Z_{\gamma(a)}}=J_{\gamma(t)}\m_{\Z_{\gamma(t)}}$, a strictly positive function. Let
\begin{equation*}
\tau_{\hat \gamma_t}:E_{\Z_{\gamma(a)}}\to E_{\Z_{\gamma(t)}}
\end{equation*}
denote parallel transport in $E$ along the curve $[0,t]\ni s\mapsto \hat\gamma_s(p)$ from $p$ to $\hat\gamma_t(p)$. Thus $\tau_{\hat \gamma_t}$ is a vector bundle morphism covering $\hat\gamma_t$ preserving the Hermitian structure. If $u$ is a section of $E$ along $\Z_{\gamma(a)}$, let $\tau_{\hat \gamma_t}u$ be the section of $E$ along $Z_{\gamma(t)}$ whose value at $\hat\gamma_t(p)$ is $\tau_{\hat \gamma_t}(u(p))$. The map
\begin{equation}\label{TheIsometry}
L^2(\Z_{\gamma(a)};E_{\Z_{\gamma(a)}}) \ni u \overset{\psi_\gamma}{\longmapsto} J_{\gamma(b)}^{1/2}\tau_{\gamma(b)}u \in L^2(\Z_{\gamma(b)};E_{\Z_{\gamma(b)}})
\end{equation}
is a unitary surjective map. 

Pick some $y_0\in \Y$. Define $\Exp_{y_0}:\Hor_{\Z_{y_0}}\to \partial\M$ as follows. If $v\in \Hor_{\Z_{y_0}}$, let $\hat \gamma_t$ be the map resulting from the horizontal lift of the geodesic $t\mapsto \exp(t\wp_* v)$ of $\Y$ of the metric $g_\Y$ and define $\Exp_{y_0}(v)=\hat \gamma_1$. Given $\eps>0$, let
\begin{equation*}
B_{y_0,\eps}=\set{v\in \Hor_{\Z_{y_0}}:g(v,v)<\eps^2}.
\end{equation*}
The hypothesis \eqref{Uniformity} implies that if $\eps$ is small enough, then $\Exp:B_{y_0,\eps}\to \partial\M$ is a diffeomorphism onto its image, $U_{y_0,\eps}$, and that $V_{y_0,\eps}=\wp(U_{y_0,\eps})$ is a geodesically convex ball about $y_0$, the image by $\exp$ of the ball in  $T_{y_0}\Y$ of radius $\eps$. The map 
\begin{equation*}
[0,1]\times U_{y_0,\eps}\ni (t,p)\mapsto \hat h(t,p)=\Exp((1-t) \Exp^{-1}(p))\in U_{y_0,\eps}
\end{equation*}
is a fiber preserving (maps fibers of $\wp$ to fibers of $\wp$) retraction of $U_{y_0,\eps}$ to $\Z_{y_0}$. In fact, with
\begin{equation*}
h:[0,1]\times V_{y_0,\eps}\to V_{y_0,\eps},\quad h(t,y)=\exp((1-t)\exp^{-1}(y)),
\end{equation*}
we have that $t\mapsto \hat h(t,p)$ is the horizontal lifting of $t\mapsto h(t,\wp(p))$. For each $y\in V_{y_0,\eps}$ the map
\begin{equation*}
\Z_y \ni p \mapsto \hat h(t,p)\in \Z_{h(t,y)}
\end{equation*}
is a diffeomorphism. The construction of the previous paragraph along the curves $t\mapsto \hat h(t,p)$, $t\in [0,1]$, gives a surjective isometry
\begin{equation*}
\psi_{y_0,y}:L^2(\Z_y;E_{\Z_y})\to L^2(\Z_{y_0};E_{\Z_{y_0}})
\end{equation*}
for each $y\in V_{y_0,\eps}$ (see \eqref{TheIsometry}) such that if $u$ is a smooth section of $E$ over $U$, then $\psi_{y_0,\cdot}(u|_{\Z_{\cdot}})$ is a smooth function on $V_{y_0,\eps}\times \Z_{y_0}$.

Now fix once and for all a point $y_\star\in \Y$ and let $\Z$ be the specific fiber $\Z_{y_\star}$ with density $\m_\Z$. Given $y_0$, construct the maps $\psi_{y_0,y}$ of the previous paragraph. Pick a continuous piecewise smooth curve $\gamma$ from $y_0$ to $y_\star$, let
\begin{equation*}
\psi_\gamma:L^2(\Z_{y_0};E_{\Z_{y_0}})\to L^2(\Z;E_\Z)
\end{equation*}
be the map associated with $\gamma$ and let
\begin{equation*}
\Psi_{y_0}:\bigsqcup_{y\in V_{y_0,\eps}} L^2(\Z_y;E_{\Z_y}) \to V_{y_0,\eps}\times L^2(\Z;E_{\Z})
\end{equation*}
be the map
\begin{equation*}
\mathfrak p^{-1}(V_{y_0,\eps}) \ni u\mapsto \big(\mathfrak p(u),\psi_\gamma\circ \psi_{y_0,\wp(u)}(u)\big).
\end{equation*}
We declare the collection of maps thus obtained by varying $y_0$, $\eps$, and $\gamma$ to be the local trivializations of the Hilbert space bundle.

Structurally, the transition functions are smooth maps 
\begin{equation*}
\chi:V\to \L(L^2(\Z;E_{\Z}))
\end{equation*}
on open sets $V\subset\Y$ into bounded linear operators on $L^2(\Z;E_\Z)$ such that $\chi(y)$ is an invertible unitary map of the form
\begin{equation*}
u\mapsto f \tau u
\end{equation*}
where $\tau:E_\Z\to E_\Z$ is bundle isometry covering a diffeomorphism $\varphi:\Z\to \Z$ and $f$ is the positive function such that $\varphi_*\m_\Z=f^2\m_\Z$. In particular, the transition functions are Fourier integral operators of a very simple kind. 

Now let $P\in \Diff^m(\partial\M;E,F)$ be an operator that commutes with multiplication by all functions $\wp^*f$ with $f\in C^\infty(\Y)$. Again let $y_0\in\Y$. Construct trivializations for the bundles associated with $E$ and $F$ using the same underlying data (metric $g$, homotopies $h$ and $\hat h$, and curve $\gamma$). Get maps
\begin{equation*}
\psi_\gamma^E\circ \psi_{y_0,y}^E,\quad \psi_\gamma^F\circ \psi_{y_0,y}^F
\end{equation*}
for each $y\in V_{y_0,\eps}$. Then (with the obvious meaning for pushforward)
\begin{equation}\label{TranslatedP}
P_y=(\psi_\gamma^F\circ \psi_{y_0,y}^F)_*\circ P\circ (\psi_\gamma^E\circ \psi_{y_0,y}^E)^*
\end{equation}
is, for each $y\in V_{y_0,\eps}$, an element of $\Diff^m(\Z;E_\Z,F_\Z)$. If $P$ is symmetric as an operator on each $\Z_y$, then so is $P_y$ for each $y$ because the transition functions are unitary.

The observation of the previous paragraph is directly applicable to the case where the operator $P$ depends on a parameter such as $\sigma\in \C$: simply replace the operator $P$ in \eqref{TranslatedP} by $P(\sigma)$. This covers in particular the case of the conormal symbol (Definition~\ref{ConSymb}) of a wedge operator since $\widehat A(\pmb\eta,\sigma)$ depends on $\pmb \eta$ only through $y=\pi_\Y\pmb\eta$. The case of the normal family $A_\wedge$ (Definition~\ref{NormalFamily}), which we now discuss, is only slightly more involved.

It is convenient to introduce at this time the pullback bundle $\wp_\wedge:\pi_\Y^*\open N_+\partial\M\to T^*\Y$,
\begin{equation}\label{NormalBdyFiberBundle}
\display{\begin{center}
\begin{picture}(0,52)
\put(-35,42){
 \put(0,0){$\pi_\Y^*\open N_+\partial\M$}
 \put(62,0){$\open N_+\partial\M$}
 \put(9,-20){\scriptsize $\wp_\wedge$}
 \put(20,-20){$\bigg\downarrow$}
 \put(15,-40){$T^*\Y$}
 \put(73,-20){$\bigg\downarrow$}
 \put(80,-20){\scriptsize $\wp_\wedge$}
 \put(72,-40){$\Y$.}
 \put(45,-32){\scriptsize $\pi_\Y$}
 \put(43,-40){$\longrightarrow$}
}
\end{picture}
\end{center}}
\end{equation}
The vertical map on the right was defined in \eqref{ModelBoundaryFibration}; the reuse of the notation $\wp_\wedge$ for the map on the left should not cause confusion. The fiber over $\pmb\eta$, $\Z^\wedge_{\pmb \eta}$, is just $\Z^\wedge_y$ with $y=\pi_\Y\pmb\eta$.

Let $\Hor_+\subset T \open N_+\partial\M$ be the horizontal bundle of a connection on the principal $\R_+$-bundle $\open N_+\partial\M\to\M$ (for example one obtained from a trivialization). Thus if the curve $\hat\gamma$ in $\partial\M$ has $\tilde \gamma$ as horizontal lifting to a curve in $\open N_+\partial\M\to\M$ starting at $v$, then its horizontal lifting to a curve starting at $\varrho v$ is $\varrho\tilde \gamma$. The property that the connection on $\open N_+\partial\M$ respects the action of $\R_+$ implies that monomials $x^{|\alpha|}$ are preserved: If $\tilde \tau_{\hat\gamma}$ denotes parallel transport along a curve $\hat \gamma$ in $\partial\M$ from the fiber of $\open N^+\partial\M$ over $p_0$ to that over $p_1$ and $x^{|\alpha|}$ is a monomial on the fiber over $p_1$, then $\tilde\tau_{\hat\gamma}^*x^{|\alpha|}$ is a monomial (of the same order) on the fiber over $p_0$.

Give the manifold $\open N_+\partial\M$ a metric which on its horizontal subbundle is the lifting of $g$ (the metric on $\partial\M$) by $\pi_\wedge$ and with which its vertical and horizontal tangent bundles are orthogonal. Write $g_\wedge$ for this Riemannian metric. Give the manifold $T^*\Y$ the Riemannian metric determined by $g_\Y$ and write $g_{T^*\Y}$ for it. Thus $\pi_\Y^*\open N_+\partial\M\to T^*\Y$, as a submanifold of $T^*\Y\times \open N_+\partial\M$, has a metric determined by those of $T^*\Y$ and $\open N_+\partial\M$. This metric is related to $g_{T^*\Y}$ in the same way as $g$ is to $g_\Y$ in \eqref{Uniformity}.

Horizontal lifts of curves in $T^*\Y$ can be described as follows. Let $\pmb \eta_0\in T^*\Y$ and $v_0\in \wp_\wedge^{-1}(\pmb\eta_0)$ and let $\gamma$ be a curve in $T^*\Y$ starting at $\pmb\eta_0$ (the map $\pi_\wedge$ was defined in \eqref{defWPwedge}). Lift the curve $\pi_\Y\circ \gamma$ to a horizontal curve in $\partial\M$ starting at $\pi_\wedge v_0$, then lift this curve to a horizontal curve $\tilde \gamma$ in $\open N_+\partial\M$ starting at $v_0$. Then $t\mapsto (\gamma(t),\tilde\gamma(t))$ is the lifting of $\gamma$ starting at $(\pmb\eta_0,v_0)$. 

The construction done before is repeated to give
\begin{equation*}
\mathfrak p_\wedge:\bigsqcup_{y\in \Y,\pmb\eta\in T^*_y\Y}L^2(\Z^\wedge_y;E^\wedge_{\Z_y^\wedge})\to T^*\Y
\end{equation*}
the structure of a Hilbert space bundle with special transition functions and model fiber $L^2(\Z^\wedge;E_{\Z^\wedge})$. The role of $\Z=\Z_{y_\star}$ earlier is now played by $\Z^\wedge_{\pmb\eta_\star}$ with some fixed element $\pmb\eta_\star\in T_{y_\star}^*\Y$. 

Given an arbitrary element $\pmb\eta_0\in T^*\Y$, we transport the operators
\begin{equation*}
A_\wedge(\pmb\eta) \in x^{-m}\Diff_b^m(\Z^\wedge_{\pmb\eta};E_{\Z^\wedge_{\pmb\eta}}^\wedge,F_{\Z^\wedge_{\pmb\eta}}^\wedge)
\end{equation*}
with $\pmb\eta$ in a suitable neighborhood $V$ of $\pmb\eta_0$ using the analogue of \eqref{TranslatedP}, thus obtaining a family
\begin{equation}\label{LocalNormalFamily}
V\ni \pmb \eta \mapsto \tilde A_{\wedge,\pmb\eta} \in x^{-m}\Diff_b^m(\Z^\wedge_{\pmb \eta};E_{\Z^\wedge_{\pmb \eta_\star}}^\wedge,F_{\Z^\wedge_{\pmb \eta_\star}}^\wedge)
\end{equation}
depending smoothly on $\pmb \eta$. The fact that the construction used a connection on $\open N_+\partial\M$ that respects the action of $\R_+$ implies that the structure of $A_\wedge(\pmb\eta)$ in the fiber variable of $\open N_+\partial\M$ is preserved, meaning the local form \eqref{WedgeSymbolLocal} is preserved. Because of \eqref{WedgeSymbolKappaHom}, changing the connection on $\open N_+\partial\M$ changes the resulting local family \eqref{LocalNormalFamily} by conjugation with $\kappa_{h(\pmb\eta)}$ and multiplication by $h(\pmb\eta)^m$ for some smooth positive $h$.

\section{The minimal domain and a semi-Fredholm theorem}\label{MinDomain}

This section is concerned with the minimal extension of an elliptic differential wedge operator (of positive order) when considered as an unbounded operator
\begin{equation}\label{Aunbounded}
A: C_c^{\infty}(\open\M;E) \subset x^{-\gamma}L^2_b(\M;E) \to x^{-\gamma}L^2_b(\M;F)
\end{equation}
for some fixed $\gamma\in\R$. We let $\Dom_{\min}=\Dom_{\min}(A)$ be the closure of $C_c^{\infty}(\open\M;E)$ with respect to the norm 
\begin{equation*}
\|u\|^2_A = \|u\|^2_{x^{-\gamma}L^2_b(\M;E)} + \|Au\|^2_{x^{-\gamma}L^2_b(\M;F)},
\end{equation*}
and denote by $A_{\min}$ the operator \eqref{Aunbounded} with domain $\Dom_{\min}$.  

The problem is essentially local in $\Y$, so we will use local coordinates $y_1,\dots,y_q$, and the notation $(y,\eta)$ for an element $\pmb\eta\in T^*\Y$. We fix some $y_0$ in the domain of the chart and choose the space $x^{-\gamma}L^2_b$ over $\Z^{\wedge}=\Z_{y_0}^{\wedge}$ as reference Hilbert space for the normal family $A_{\wedge}(\pmb\eta)=A_{\wedge}(y,\eta)$. We will write $E^\wedge$ when we mean $E^\wedge_\Z$. Our starting point is thus the unbounded operator
\begin{equation}\label{NormalFamilyUnbounded}
A_{\wedge}(y,\eta) : C_c^{\infty}(\Z^{\wedge};E^\wedge) \subset
x^{-\gamma}L^2_b(\Z^{\wedge};E^\wedge) \to x^{-\gamma}L^2_b(\Z^{\wedge};F^\wedge).
\end{equation}

Let $H_b^m(\Z^\wedge;E^\wedge)$ be the space of functions $u\in L^2_b(\Z^\wedge;E^\wedge)$ such that $Pu\in L^2_b$ for each $P\in \Diff_b^m(\Z^\wedge;E^\wedge)$ with coefficients independent of $x$ (i.e., $[P,\nabla_{x\partial_{x}}]=0$). Recall that $\Z^\wedge$ is a principal $\R_+$ bundle over $\Z$. Let $H_{\rm cone}^s({\Z^\wedge};E^\wedge)$ be the space of $E^\wedge$-valued distributions $u$ on $\Z^\wedge$ such that, given any coordinate patch $U\subset \Z$ diffeomorphic to an open subset of the sphere $S^{n}$ such that $E^\wedge_U\cong U\times\C^r$,  and any $\chi \in C_c^\infty(U)$, we have that $\chi u$ is the restriction to $\R^{n+1}\minus\{0\}$ of an element of $H^s(\R^{n+1};\C^r)$; here $\R_+ \times S^{n}$ is identified with $\R^{n+1} \minus \{0\}$ via polar coordinates.

\medskip
The following proposition follows from standard results about cone operators; see for example \cite{GiMe01, Le97, SchuNH}.
\begin{proposition}\label{WedgeMinProperties}
Let $A \in x^{-m}\Diff^m_{e}(\M;E,F)$ be $w$-elliptic.
\begin{enumerate}[$(i)$]
\item For $(y,\eta) \in T^*\Y$ the closure of $C_c^{\infty}(\Z_y^{\wedge};E^{\wedge})$ with respect to the norm
\begin{equation*}
 \|u\|_{A_{\wedge}(y,\eta)}^2 = \|u\|^2_{x^{-\gamma}L^2_b} + \|A_{\wedge}(y,\eta)u\|^2_{x^{-\gamma}L^2_b}
\end{equation*}
is independent of $\eta$. In other words, the minimal domain of $A_{\wedge}(y,\eta)$ depends only on $y\in\Y$, hence we will denote it by $\Dom_{\wedge,\min}(y)$.
\item If $\spec_{b,y}(A) \cap \{\sigma\in \C \st \Im(\sigma)=\gamma-m\}= \emptyset$, then for any cut-off function $\omega$ with $\omega=1$ near $x=0$, we have
\begin{equation*}
\Dom_{\wedge,\min}(y) = 
\omega x^{-\gamma+m}H^m_b(\Z_y^{\wedge};E^{\wedge}) + (1-\omega) x^{\frac{n+1}{2}-\gamma} H^m_{\textup{cone}}(\Z_y^{\wedge};E^{\wedge}).
\end{equation*}
\item The group of isometries $\kappa_{\varrho} \in
\L(x^{-\gamma}L^2_b(\Z_y^{\wedge}; E^{\wedge}))$, $\varrho > 0$, defined in \eqref{kappadef}, restricts to a strongly continuous group on $\Dom_{\wedge,\min}(y)$.
\item For $(y,\eta) \in T^*\Y\setminus 0$, the normal family
$$
A_{\wedge}(y,\eta) : \Dom_{\wedge,\min}(y) \to x^{-\gamma}L^2_b(\Z_y^{\wedge};F^{\wedge})
$$
is Fredholm.
\end{enumerate}
\end{proposition}

The goal of this section is to prove the following theorem.

\begin{theorem}\label{DminTheorem}
Let $A \in x^{-m}\Diff^m_e(\M;E,F)$ be $w$-elliptic. If 
\begin{equation*}
\pi_{\C}\spec_{e}(A) \cap \{\sigma \in \C \st \Im\sigma = \gamma-m\}= \emptyset
\end{equation*}
and the normal family $A_{\wedge}(y,\eta) : \Dom_{\wedge,\min}(y) \subset x^{-\gamma}L^2_b \to x^{-\gamma}L^2_b$ is injective for all $(y,\eta) \in T^*\Y \setminus 0$, then
\begin{equation*}
\Dom_{\min}(A) = x^{-\gamma+m}H^m_{e}(\M;E).
\end{equation*}
Moreover, 
\begin{equation*} 
A : \Dom_{\min} \to x^{-\gamma}L^2_b(\M;F)
\end{equation*}
is a semi-Fredholm operator with finite-dimensional kernel and closed range.
\end{theorem}
Note that, in view of the $\kappa$-homogeneity of $A_{\wedge}$, it is sufficient to only check the injectivity of $A_{\wedge}(y,\eta)$ on the cosphere bundle $S^*\Y$.

\begin{example}\label{RegularEllipticmin}
Assume that $\Z=\{\text{pt}\}$, so $\Y = \partial\M$. Let $A\in \Diff^m(\M;E)$ be a regular elliptic operator on $\M$. Let $L^2(\M;E)$ be the $L^2$-space of sections of $E$ with respect to any smooth density on $\M$ and Hermitian metric on $E$, and consider $A$ as an unbounded operator
\begin{equation*}
A : C_c^{\infty}(\open\M;E) \subset L^2(\M;E) \to L^2(\M;E).
\end{equation*}
We have $L^2(\M;E) = x^{-1/2}L^2_b(\M;E)$, and the operator $A$ is a $w$-elliptic element of $x^{-m}\Diff^m_{e}(\M;E)$ with
\begin{equation*}
\spec_{e}(A) = \Y \times \{0,-i,\ldots,-(m-1)i\}.
\end{equation*}
In fact,
\begin{equation*}
\widehat{A}(y,\sigma) = \varphi(y) \sigma(\sigma+i)\cdots(\sigma+(m-1)i)
\end{equation*}
for some invertible $\varphi \in C^{\infty}(\Y;\Hom(E_{\Y}))$. Moreover, $A_{\wedge}(y,\eta)$ is a family of ordinary differential operators in the inward normal direction, parametrized by the cotangent bundle $T^*\Y$. More precisely, if
\begin{align*}
A &= \sum\limits_{k+|\alpha| \leq m} a_{k,\alpha}(x,y)D_y^{\alpha}D_x^k \\
\intertext{in coordinates $y_1,\ldots,y_{d-1}$ and $x$ near a boundary point, then}
A_{\wedge}(y,\eta) &= \sum\limits_{k+|\alpha| = m} a_{k,\alpha}(0,y)\eta^{\alpha}D_x^k.
\end{align*}
In particular,
\begin{equation*}
A_{\wedge}(y,\eta) : \Dom_{\wedge,\min} \subset x^{-1/2}L_b^2(\R_+;E^{\wedge}) \to x^{-1/2}L_b^2(\R_+;E^{\wedge})
\end{equation*}
is injective for $\eta \neq 0$ since $\Dom_{\wedge,\min} = H_0^m(\R_+;E^{\wedge})$. In view of Theorem~\ref{DminTheorem}, the minimal domain of $A$ in $L^2(\M;E)$ is
\begin{equation*}
\Dom_{\min}= x^{m-1/2}H^m_{e}(\M;E) = H^m_0(\M;E),
\end{equation*}
and $A : \Dom_{\min} \to L^2(\M;E)$ is semi-Fredholm.
\end{example}

\medskip
\begin{proof}[Proof of Theorem~\ref{DminTheorem}]
The proof is based on the parametrix construction performed in Section~\ref{ParametrixConstruction}. There we will show that under the assumptions of the theorem, there exists a left parametrix
\begin{equation*}
B : x^{-\gamma}L^2_b(\M;F) \to x^{-\gamma+m}H^m_{e}(\M;E)
\end{equation*}
of $A$ modulo compact remainders; see Remark~\ref{AParametrix}. Therefore, in particular,
\begin{equation*}
A : x^{-\gamma+m}H^m_{e}(\M;E) \to x^{-\gamma}L^2_b(\M;F)
\end{equation*}
is semi-Fredholm with finite-dimensional kernel and closed range. As a consequence of this and Proposition~\ref{ClosedFredholmDomains}, we get that $x^{-\gamma+m}H^m_{e}(\M;E)$ is complete in the graph norm of $A$. Thus $\Dom_{\min}(A) = x^{-\gamma+m}H^m_{e}(\M;E)$ since $C_c^{\infty}(\open\M;E)$ is dense in $x^{-\gamma+m}H^m_{e}(\M;E)$.
\end{proof}

\begin{proposition}\label{ClosedFredholmDomains}
Let $H_1$ and $H_2$ be Hilbert spaces, and let $D$ be a Banach subspace of $H_1$ equipped with a norm $\|\cdot\|_D$ such that the inclusion map $(D,\|\cdot\|_D) \to (H_1,\|\cdot\|_{H_1})$ is continuous. Let $A : D\subset H_1 \to H_2$ be continuous and assume that:
\begin{enumerate}[$(i)$]
\item The range of $A$ is closed.
\item The kernel $N(A) \subset H_1$ is closed with respect to $\|\cdot\|_{H_1}$.
\end{enumerate}
Then the operator $A$ with domain $D$ is closed, i.e., $D$ is complete with respect to the graph norm $\|u\|_A = \|u\|_{H_1} + \|Au\|_{H_2}$.
\end{proposition}
Note that assumptions $(i)$ and $(ii)$ are fulfilled if $A$ is a semi-Fredholm operator with finite dimensional kernel (thus, in particular, if $A : D \to H_2$ is Fredholm).

\begin{proof}
Let $\Pi = \Pi^2 \in \L(H_1)$ be a projection onto $N(A)$. Since $\Pi(D) = N(A) \subset D$, and since the inclusion map $(D,\|\cdot\|_D) \to (H_1,\|\cdot\|_{H_1})$ is continuous, we conclude that the graph of the restriction $\Pi : D \to D$ is closed. Hence $\Pi \in \L(D)$ by the closed graph theorem (note that $D$ is complete with respect to $\|\cdot\|_D$). So $\Pi$ is a continuous projection onto $N(A)$, both on $H_1$ and on $D$. Write
$$
D = N(A) \oplus (I-\Pi)(D), \quad H_2 = X \oplus R(A).
$$
Then $A|_{(I-\Pi)(D)} : (I-\Pi)(D) \to R(A)$ is a topological isomorphism. Let
$$
P = \bigl(A|_{(I-\Pi)(D)}\bigr)^{-1}\circ Q : H_2 \to D,
$$
where $Q = Q^2 \in \L(H_2)$ is the projection onto $R(A)$ according to the decomposition $H_2 = X \oplus R(A)$. Then $P \in \L(H_2,D)$ and $PA = I-\Pi$ on $D$.

Now assume that $(u_k)_k \subset D$ is such that $u_k \to u$ in $H_1$ and $Au_k \to v$ in $H_2$. We have to show that $u \in D$ and that $v = Au$. Since $A : D \to H_2$ has closed range, we conclude that there is some $\bar{u} \in D$ such that $Au_k \to A\bar{u} = v$ in $H_2$. Consequently,
\begin{equation*}
PAu_k = (I-\Pi)u_k \to PA\bar{u} = (I-\Pi)\bar{u} \;\text{ in } D.
\end{equation*}
But since $D \hookrightarrow H_1$, we also have $(I-\Pi)u_k \to (I-\Pi)\bar{u}$ in $H_1$. Since $(I-\Pi) \in \L(H_1)$ and $u_k \to u$ in $H_1$, we conclude that $(I-\Pi)u = \lim\limits_{k \to \infty}(I-\Pi)u_k = (I-\Pi)\bar{u}$. Thus
\begin{align*}
u &= \Pi u + (I - \Pi)u = \Pi u + (I - \Pi)\bar{u} \in D, 
\intertext{and moreover,}
Au &= A\Pi u + A(I-\Pi)u = A(I-\Pi)\bar{u} = A\bar{u} = v.
\end{align*}
\end{proof}

\section{The resolvent family}\label{Sec-ResolventFamily}

Let $A$ be an elliptic wedge operator of order $m > 0$. Let $\Lambda$ be a ray or a closed sector properly contained in $\C$. Consider the operator family
\begin{equation*}
A - \lambda : \Dom_{\min}(A) \to x^{-\gamma}L^2_b(\M;E).
\end{equation*}
Our goal is the following theorem.

\begin{theorem}\label{MinLeftInverseandExpansion}
Let $A \in x^{-m}\Diff^m_{e}(\M;E)$ with $m > 0$ be such that
\begin{equation*}
\spec(\wsym(A)) \cap \Lambda = \emptyset \;\text{ on } \wT^*\M \setminus 0,
\end{equation*}
and assume that
\begin{equation*}
\pi_{\C}\spec_{e}(A) \cap \{\sigma \in \C \st \Im(\sigma) = \gamma-m\} = \emptyset.
\end{equation*}
If the normal family $A_{\wedge}(y,\eta) - \lambda : \Dom_{\wedge,\min}(y)\subset x^{-\gamma}L^2_b \to x^{-\gamma}L^2_b$ is injective for all $(y,\eta,\lambda) \in (T^*\Y\times\Lambda)\setminus 0$, then
\begin{equation}\label{Aminminuslambda}
A - \lambda : \Dom_{\min}(A) \to x^{-\gamma}L^2_b(\M;E)
\end{equation}
is injective for $\lambda \in \Lambda$ with $|\lambda|$ large, and there exists a parametrix
\begin{equation*}
B(\lambda) : x^{-\gamma}L^2_b(\M;E) \to \Dom_{\min}(A)
\end{equation*}
of $A-\lambda$ with the following properties. For $|\lambda|$ large, $B(\lambda)$ is a left inverse of \eqref{Aminminuslambda}, and for $\alpha = (\alpha_1,\alpha_2) \in \N_0^2$ with $1+|\alpha| > \frac{d}{m}$,
the operator family
\begin{equation*}
\partial_{\lambda}^{\alpha_1}\partial_{\overline{\lambda}}^{\alpha_2}B(\lambda) :
x^{-\gamma}L^2_b(\M;E) \to x^{-\gamma}L^2_b(\M;E)
\end{equation*}
is of trace class. Moreover, for every $\varphi \in C^{\infty}(\M;\Hom(E))$, the trace
\begin{equation*}
t(\lambda) = \Tr \bigl(\varphi \partial_{\lambda}^{\alpha_1}\partial_{\overline{\lambda}}^{\alpha_2} B(\lambda)\bigr): \Lambda \to \C
\end{equation*}
is a symbol on $\Lambda$ having an asymptotic expansion
\begin{equation}\label{tlambdaexpansion}
 t(\lambda) \sim \sum_{j=0}^{\infty}\sum_{k=0}^{m_j}t_{jk}(\lambda)\log^k(\lambda) \text{ as } |\lambda| \to \infty
\end{equation}
with functions $t_{jk}$ that satisfy the homogeneity relation
\begin{equation*}
t_{jk}(\varrho\lambda) = \varrho^{\frac{d-j}{m}-1-|\alpha|}t_{jk}(\lambda) \text{ for } \varrho > 0.
\end{equation*}
Here $m_j \leq 1$ for all $j$, and $m_j = 0$ for $j \leq n=\dim \Z$. Recall that $d = \dim\M$.
\end{theorem}

The proof of this theorem is given in Section~\ref{ExpansionProof}. It relies on the parametrix construction presented in Section~\ref{ParametrixConstruction}. 

The logarithm in formula \eqref{tlambdaexpansion} is a holomorphic branch of the logarithm with respect to some cut $\Gamma \not\subset \Lambda$. The powers and logarithms in Corollary~\ref{Resolventexpansion} below are to be understood in the same way.

\begin{corollary}\label{Resolventexpansion}
Let $A$ satisfy all the assumptions of Theorem~\ref{MinLeftInverseandExpansion}. If in addition the normal family
\begin{equation*}
A_{\wedge}(y,\eta) - \lambda : \Dom_{\wedge,\min}(y) \subset x^{-\gamma}L^2_b \to x^{-\gamma}L^2_b
\end{equation*}
is surjective $($and therefore bijective$)$ for all $(y,\eta,\lambda) \in (T^*\Y\times\Lambda)\setminus 0$, then $\Lambda$ is a ray/sector of minimal growth for $A_{\min}$. In other words, the family $A_{\min}-\lambda$ is invertible for $\lambda \in \Lambda$ with $|\lambda|$ large, and
\begin{equation*}
\|\bigl(A_{\min}-\lambda\bigr)^{-1}\|_{\L(x^{-\gamma}L^2_b)} = O(|\lambda|^{-1}) \text{ as } |\lambda| \to \infty,\; \lambda \in \Lambda.
\end{equation*}
Moreover, for every $\ell \in \N$ with $\ell > \frac{d}{m}$,
\begin{equation*}
\bigl(A_{\min}-\lambda\bigr)^{-\ell} : x^{-\gamma}L^2_b(\M;E) \to x^{-\gamma}L^2_b(\M;E)
\end{equation*}
is of trace class, and for every $\varphi \in C^{\infty}(\M;\Hom(E))$ we have an expansion
\begin{equation*}
\Tr\bigl(\varphi\bigl(A_{\min}-\lambda\bigr)^{-\ell}\bigr) \sim
\sum_{j=0}^{\infty}\sum_{k=0}^{m_j} \alpha_{jk}\lambda^{\frac{d-j}{m}-\ell}\log^k(\lambda) \text{ as } |\lambda| \to \infty.
\end{equation*}
Here $m_j \leq 1$ for all $j$, and $m_j = 0$ for $j \leq n$.
\end{corollary}

\begin{proof}[Proof of the corollary]
Since the normal family $A_{\wedge}(y,\eta) - \lambda$ is assumed to be bijective on $\Dom_{\wedge,\min}(y)$, the parametrix $B(\lambda)$ constructed in Section~\ref{ParametrixConstruction} is an actual inverse of $A_{\min} - \lambda$ for large $\lambda \in \Lambda$; see the comments right after Theorem~\ref{ParametrixTheorem}. Thus, for $\lambda$ large, the resolvent $(A_{\min}-\lambda)^{-1}$ coincides with $B(\lambda)$ and the asymptotic expansion \eqref{tlambdaexpansion} applies to the $\lambda$-derivatives of the resolvent.   
Note that the analyticity of the resolvent implies the analyticity of the components of the trace expansion. Thus the asymptotic expansion can indeed be written in terms of powers of $\lambda$ and logarithms, as claimed.

Finally, that $\Lambda$ is a ray/sector of minimal growth for $A_{\min}$ follows immediately from the structure of the parametrix $B(\lambda)$ described in Section~\ref{ParametrixConstruction}.
\end{proof}

\section{Parametrix construction}\label{ParametrixConstruction}

Let $\Lambda$ be a closed sector in $\C$. Given $A \in x^{-m}\Diff^m_{e}(\M;E)$ with $m>0$, our main goal here is to construct a parametrix $B(\lambda)$ to
\begin{equation*}
A-\lambda: x^{-\gamma+m}H^m_{e}(\M;E) \to x^{-\gamma}L^2_b(\M;E), \ \lambda\in\Lambda,
\end{equation*}
under the assumptions and with the properties listed in Theorem~\ref{MinLeftInverseandExpansion}. 
The methods involved in our construction are known. We will follow the approach developed by Schulze for the study of pseudodifferential operators on manifolds with edge singularities; see e.g. \cite{SchuNH}. This approach is the most direct and closest to what we need in this paper. 

\subsection*{Assumptions.}
Throughout this section we will assume:
\begin{center}
\parbox{.9\textwidth}{
\begin{enumerate}[1.]
\item $\spec(\wsym(A))\cap \Lambda=\emptyset$ on $\wT^*\M\setminus 0$ \ ($w$-ellipticity with parameter).
\item $\pi_{\C}\spec_{e}(A) \cap \{\sigma \in \C \st \Im \sigma = \gamma-m\}= \emptyset$.
\item The normal family
$$
A_{\wedge}(y,\eta) - \lambda : \Dom_{\wedge,\min}(y) \subset x^{-\gamma}L^2_b \to x^{-\gamma}L^2_b
$$
is injective for all $(y,\eta,\lambda) \in (T^*\Y\times\Lambda)\setminus 0$.
\end{enumerate}
} 
\end{center}
Under these assumptions, we will follow essentially three steps to construct our parametrix $B(\lambda)$. In the first two steps we make use of the $w$-ellipticity with parameter and the assumption on the $e$-spectrum to find a parametrix of $A-\lambda$ modulo operators that are smoothing on $\open\M$. We then use the third assumption to invert $A-\lambda$ from the left modulo compact remainders for any $\lambda$, and to obtain an actual left inverse for large $|\lambda|$. 

\medskip
In adapted local coordinates near the boundary of $\M$, the complete symbol of $A-\lambda$ has the form
\begin{equation*}
x^{-m}\biggl(\,\sum_{k+|\alpha|+|\beta|\leq m} a_{k\alpha\beta}(x,y,z)(x\xi)^k (x\eta)^\alpha \zeta^\beta -x^m\lambda\biggr)
\end{equation*}
with coefficients $a_{k\alpha\beta}$ that are smooth up to $x=0$. The assumed invertibility of $\wsym(A)-\lambda$ allows the construction of a parametrix $B_1(\lambda)$ whose complete symbols are locally of the form
$$
x^{m}b(x,y,z,x\xi,x\eta,\zeta;x^m\lambda),
$$
where $b(x,y,z,\xi,\eta,\zeta;\lambda)$ is a classical symbol of order $-m$. This means that each $b$ is a smooth function satisfying the estimates
$$
\Bigl|\partial_{(x,y,z)}^{\alpha}\partial_{(\xi,\eta,\zeta)}^{\beta}
\partial_{\lambda}^{\gamma_1}\partial_{\overline{\lambda}}^{\gamma_2} b\Bigr| = 
O\Bigl((1 + |(\xi,\eta,\zeta)| + |\lambda|^{1/m})^{-m-|\beta|-m|\gamma|}\Bigr)
$$
as $(1 + |(\xi,\eta,\zeta)| + |\lambda|^{1/m}) \to \infty$ for every $\alpha,\beta\in\N_0^{d}$ and $\gamma\in\N_0^2$, locally uniformly in $(x,y,z)$ up to the boundary. Moreover, $b$ admits an asymptotic expansion
\begin{equation*}
b(x,y,z,\xi,\eta,\zeta;\lambda) \sim \sum_{j=0}^{\infty}b_{j}(x,y,z,\xi,\eta,\zeta;\lambda)
\end{equation*}
with components $b_{j}$ that are anisotropic homogeneous of degree $-m-j$, i.e.,
\begin{equation*}
b_{j}(x,y,z,\varrho\xi,\varrho\eta,\varrho\zeta;\varrho^m\lambda) = 
\varrho^{-m-j}b_{j}(x,y,z,\xi,\eta,\zeta;\lambda) \text{ for } \varrho > 0.
\end{equation*}

In general, $B_1(\lambda) : C_c^{\infty}(\open\M;E) \to C^{\infty}(\open\M;E)$ does not extend by continuity to an operator on $x^{-\gamma}L^2_b$. However,  it can be modified by a smoothing operator (using Mellin quantization) in order to get a parametrix, also denoted by $B_1(\lambda)$, such that
\begin{equation*} 
B_1(\lambda): x^{-\gamma}L^2_b(\M;E) \to x^{-\gamma+m}H^m_{e}(\M;E)
\end{equation*}
is bounded. The resulting operator $B_1(\lambda)$ is obtained, near the boundary,  by glueing together operators of the form $\Op_y(q)(\lambda)$ with an operator-valued symbol $q(y,\eta;\lambda): C_c^{\infty}(\Z^{\wedge};E^\wedge) \to C^{\infty}(\Z^{\wedge};E^\wedge)$ of the form
\begin{equation}\label{Completeedgesymbols}
q(y,\eta;\lambda) = \opm(h)(y,\eta;\lambda) + v(y,\eta;\lambda) + g(y,\eta;\lambda)
\end{equation}
with components $\opm(h)$, $v$, and $g$, defined as follows. 

First, for $u \in C_c^{\infty}(\Z^{\wedge};E^\wedge)$,
\begin{equation}\label{HolomMellinOp}
\opm(h)(y,\eta;\lambda)u = \frac{1}{2\pi} \int\limits_{\Im\sigma=\gamma}\!\int_{\overline{\R}_+}
\Bigl(\frac{x}{x'}\Bigr)^{i\sigma}x^{m}h(x,y,\sigma,x\eta;x^m\lambda)u(x') \frac{dx'}{x'}\,d\sigma,
\end{equation}
where $h(x,y,\sigma,\eta;\lambda)$ is a so-called Mellin symbol; it is smooth in $y$ and $x$ up to $x = 0$, it is holomorphic in $\sigma \in \C$, and it takes values in the pseudodifferential operators of order $-m$ on the fiber $\Z$ that depend isotropically on the parameters $(\Re\sigma,\eta)$ and anisotropically on the parameter $\lambda$. 

Let $|\eta,\lambda|_m$ denote the anisotropic distance function  
\begin{equation}\label{rhosubm}
|\eta,\lambda|_m = \bigl(|\eta|^{2m} + |\lambda|^2\bigr)^{\frac{1}{2m}},
\end{equation}
and let $\varrho_m:\R^q\times\Lambda\to\R_+$ be a positive smooth function such that
\begin{equation*}
\varrho_m(\eta,\lambda) = |\eta,\lambda|_m \text{ for }  |\eta,\lambda|_m>1.
\end{equation*}
For the second component of $q(y,\eta;\lambda)$ in \eqref{Completeedgesymbols}, we have 
\begin{equation}\label{SmoothingMellin}
v(y,\eta;\lambda)u=\frac{1}{2\pi}\int\limits_{\Im\sigma=\gamma}\!\int_{\overline{\R}_+}
\Bigl(\frac{x}{x'}\Bigr)^{i\sigma}x^{m} h_v(x,x',y,\sigma,\eta;\lambda)u(x')\frac{dx'}{x'}\,d\sigma
\end{equation}
with $h_v(x,x',y,\sigma,\eta;\lambda)=\omega\bigl(x\varrho_m(\eta,\lambda)\bigr)\tilde{h}(y,\sigma)\omega\bigl(x'\varrho_m(\eta,\lambda)\bigr)$, where $\omega \in C^{\infty}(\overline{\R}_+)$ is some cut-off function supported near zero, and $\tilde{h}(y,\sigma)$ is a family of smoothing operators on $\Z$, depending smoothly on $y$, holomorphic in $\sigma$ in a small strip centered at $\Im\sigma=\gamma$, and rapidly decreasing in the parameter $\Re\sigma$ as $|\Re\sigma|\to \infty$. 

In order to describe the last component of $q(y,\eta;\lambda)$, we need to introduce the following weighted spaces: For $s,t,\delta\in\R$, let $\K_\delta^{s,t}(\Z^\wedge;E^\wedge)$ be the space consisting of all $u$ such that $\omega u \in x^{t}H^s_b(\Z^{\wedge};E^\wedge)$ and $(1-\omega)u\in x^{\frac{n+1}{2}-\gamma-\delta} H^s_{\textup{cone}}(\Z^{\wedge};E^\wedge)$. Observe that under the conditions of Proposition~\ref{WedgeMinProperties}, we have $\Dom_{\wedge,\min}(y)=\K_0^{m,-\gamma+m}$.

Now, $g(y,\eta;\lambda)$ in \eqref{Completeedgesymbols} is a smooth family of operators in
$\L\bigl(\K_{\delta}^{s,-\gamma},\K_{\delta'}^{s',-\gamma+m+\varepsilon}\bigr)$ for any $s,s',\delta,\delta' \in \R$ and some sufficiently small $\varepsilon > 0$, which for all multi-indices $\alpha$, $\beta$, $\gamma$, satisfies the norm estimates 
\begin{equation}\label{GreenSymbolEstimate}
\Bigl\| \kappa^{-1}_{\varrho_m(\eta,\lambda)} \bigl( \partial_y^{\alpha}\partial_{\eta}^{\beta} 
\partial_{\lambda}^{\gamma_1} \partial_{\overline{\lambda}}^{\gamma_2} 
g(y,\eta;\lambda)\bigr) \kappa_{\varrho_m(\eta,\lambda)}\Bigr\| 
= O\bigl(\varrho_m(\eta,\lambda)^{\mu-|\beta|-m|\gamma|}\bigr)
\end{equation}
as $\varrho_m(\eta,\lambda) \to \infty$, with $\mu=-m$, locally uniformly for $y \in \Omega$. Here $\kappa_{\varrho}$ is given by $\kappa_{\varrho}u(x,z) = \varrho^{\gamma}u(\varrho x,z)$ for $\varrho > 0$, see \eqref{kappadef}.  Moreover, there is an asymptotic expansion
\begin{equation}\label{GreenAsymptotic}
g(y,\eta;\lambda) \sim \sum\limits_{j=0}^{\infty}g_{j}(y,\eta;\lambda)
\end{equation}
with the $g_{j}(y,\eta;\lambda)$ satisfying the $\kappa$-homogeneity relations
\begin{equation}\label{Greenkappahomogeneous}
g_{j}(y,\varrho\eta;\varrho^m\lambda) = \varrho^{\mu-j}\kappa_{\varrho}\,
g_{j}(y,\eta;\lambda)\kappa_{\varrho}^{-1} \textup{ for } \varrho > 0.
\end{equation}
The formal adjoint of $g(y,\eta;\lambda)$ with respect to the $x^{-\gamma}L^2_b$-inner product on $\Z^{\wedge}$ has similar properties.

The quantized family $\Op_y(v+g)(\lambda)$ with $v$ and $g$ as in \eqref{Completeedgesymbols} give rise to smoothing operators on $\open\M$ whose Schwartz kernels may exhibit blow-up at the boundary. The final parametrix $B(\lambda)$ will be obtained via suitable modifications of $B_1(\lambda)$ by smoothing operators of that kind, modulo global negligible remainders. By negligible we mean a family that takes values in $\ell^1\bigl(x^{-\gamma}L^2_b, x^{-\gamma+m}H^m_{e}\bigr)$ and is rapidly decreasing in $\lambda \in \Lambda$ as $|\lambda| \to \infty$.

\begin{remark}\label{DerivativeRemark}
Let $\alpha = (\alpha_1,\alpha_2) \in \N_0^2$, $|\alpha| > 0$, and let $\varphi \in C^{\infty}(\M;\Hom(E))$. For any parametrix $B_1(\lambda)$ with local structure as described above, we have that modulo negligible operators, the family $\varphi\partial_{\lambda}^{\alpha_1} \partial_{\overline{\lambda}}^{\alpha_2}B_1(\lambda)$ can be realized (near the boundary) as $\Op_y(q_\alpha)(\lambda)$ for some $q_\alpha(y,\eta;\lambda)$ as in \eqref{Completeedgesymbols} with $v(y,\eta;\lambda)=0$ and $g(y,\eta;\lambda)$ such that \eqref{GreenSymbolEstimate} and \eqref{Greenkappahomogeneous} are satisfied with $\mu=-m-m|\alpha|$. In this case, the corresponding Mellin symbol $h$ is of order $-m-m|\alpha|$ and $\opm(h)(y,\eta;\lambda)$ has an extra $x^{m|\alpha|}$ inside the integral in \eqref{HolomMellinOp}.
\end{remark}

We continue with the construction of the parametrix $B(\lambda)$ and now proceed to refine $B_1(\lambda)$ making use of our assumption on the boundary spectrum of $A$. Notice that up to this point we only made use of the $w$-ellipticity with parameter. As described above, near the boundary we have $B_1(\lambda) = \Op_y(q)(\lambda)$ modulo negligible terms. With $h$ as in \eqref{HolomMellinOp} and $h_v$ as in \eqref{SmoothingMellin}, the conormal symbol of $B_1(\lambda)$ can be represented as
\begin{equation*}
\widehat{B}_1(y,\sigma) = h(0,y,\sigma,0,0) + h_v(0,0,y,\sigma,0,0).
\end{equation*}
For every $y \in \Y$, $\widehat{B}_1(y,\sigma) : C^{\infty}(\Z_y;E) \to C^{\infty}(\Z_y;E)$ is a holomorphic family of pseudodifferential operators defined for $\sigma$ in a small strip centered at $\Im\sigma = \gamma$. Moreover, by construction we have that $\widehat{B}_1(y,\sigma+im)$ is a parameter-dependent parametrix of $\widehat{A}(y,\sigma)$ for $\Im\sigma = \gamma - m$, modulo remainder terms that are smoothing on the fibers $\Z_y$ and rapidly decreasing as $|\Re\sigma| \to \infty$. Because of our standing assumption on $\spec_e(A)$, we have that $\widehat{A}(y,\sigma)$ is invertible for $\Im\sigma = \gamma-m$.  Thus we can locally modify $B_1(\lambda)$ by terms of the form $\Op_y(v)(\lambda)$ with $v(y,\eta;\lambda)$ as in \eqref{Completeedgesymbols} in such a way that $\widehat{B}_1(y,\sigma+im)$ is an inverse of $\widehat{A}(y,\sigma)$ for all $\Im\sigma = \gamma-m$. Denote this modification by $B_2(\lambda)$. We have now used the first two assumptions at the beginning of this section.

Modulo negligible terms, we locally have $B_2(\lambda) = \Op_y(q_2)(\lambda)$ with $q_2(y,\eta;\lambda)$ as in \eqref{Completeedgesymbols}. Thus, for $(\eta,\lambda) \neq (0,0)$, the normal family $B_{2,\wedge}(y,\eta;\lambda)$ of $B_2(\lambda)$ can be locally represented as
\begin{equation}\label{B2wedge}
B_{2,\wedge}(y,\eta;\lambda) = \opm(h_0)(y,\eta;\lambda) + v_{0}(y,\eta;\lambda) + g_{0}(y,\eta;\lambda),
\end{equation}
where $\opm(h_0)(y,\eta;\lambda)$ is of the same form as \eqref{HolomMellinOp} with $h(x,y,\sigma,x\eta;x^m\lambda)$ replaced by $h(0,y,\sigma,x\eta;x^m\lambda)$, $v_{0}(y,\eta;\lambda)$ is of the form \eqref{SmoothingMellin} with the function $\varrho_m(\eta,\lambda)$ in $h_v$ replaced by $|\eta,\lambda|_m$ from \eqref{rhosubm}, and where $g_{0}(y,\eta;\lambda)$ is the leading term in the asymptotic expansion \eqref{GreenAsymptotic} of $g(y,\eta;\lambda)$. Observe that we have the $\kappa$-homogeneity
\begin{equation*}
B_{2,\wedge}(y,\varrho\eta;\varrho^m\lambda) = \varrho^{-m}\kappa_{\varrho}
B_{2,\wedge}(y,\eta;\lambda)\kappa_{\varrho}^{-1} \text{ for } \varrho > 0.
\end{equation*}
Moreover, on $(T^*\Y\times\Lambda)\setminus 0$, $B_{2,\wedge}(y,\eta;\lambda)$ maps $x^{-\gamma}L^2_b(\Z_y^{\wedge};E^{\wedge})$ to $\Dom_{\wedge,\min}(y)$ and it inverts the normal family
\begin{equation}\label{Awedgeminuslambda}
A_{\wedge}(y,\eta) - \lambda: \Dom_{\wedge,\min}(y)\to x^{-\gamma}L^2_b(\Z_y^{\wedge};E^{\wedge})
\end{equation}
modulo compact operators. 

As a last step, we now make use of the assumed injectivity of \eqref{Awedgeminuslambda} to obtain our final refinement of the parametrix. Let $S(T^*\Y\times \C)$ be the sphere bundle of $T^*\Y\times \C\to \Y$ with respect to some metric and let $S_\Lambda=S(T^*\Y\times \C)\cap (T^*\Y\times \Lambda)$. The injectivity of the normal family implies that there is a smooth vector bundle $J \to S_\Lambda$ and a smooth family
\begin{equation*}
k_{\wedge}(y,\eta;\lambda) : J_{(y,\eta,\lambda)} \to C_c^{\infty}(\Z_y^{\wedge};E^{\wedge})
\end{equation*}
such that
\begin{equation}\label{AwedgeK}
\begin{pmatrix} A_{\wedge}(y,\eta) - \lambda & k_{\wedge}(y,\eta;\lambda) \end{pmatrix} :
\begin{array}{c}
\Dom_{\wedge,\min}(y) \\ \oplus \\ J_{(y,\eta,\lambda)}
\end{array}
\to x^{-\gamma}L^2_b
\end{equation}
is invertible on $S_\Lambda$. We extend $k_{\wedge}(y,\eta;\lambda)$ by $\kappa$-homogeneity to all of $\bigl(T^*\Y\times\Lambda\bigr)\setminus 0$ such that
\begin{equation*}
k_{\wedge}(y,\varrho\eta;\varrho^m\lambda) = \varrho^{m}\kappa_{\varrho} k_{\wedge}(y,\eta;\lambda) \textup{ for } \varrho > 0.
\end{equation*}
The inverse of \eqref{AwedgeK} is then of the form
\begin{equation*}
\begin{pmatrix} A_{\wedge}(y,\eta) - \lambda & k_{\wedge}(y,\eta;\lambda)\end{pmatrix}^{-1} =
\begin{pmatrix}
B_{2,\wedge}(y,\eta;\lambda) + g_{2,\wedge}(y,\eta;\lambda) \\ t_{\wedge}(y,\eta;\lambda)
\end{pmatrix},
\end{equation*}
where $g_{2,\wedge}(y,\eta;\lambda)$ satisfies the $\kappa$-homogeneity \eqref{Greenkappahomogeneous} with $\mu = -m$ and $j = 0$. 

Let $\chi \in C^{\infty}(T^*\Y\times\Lambda)$ be such that $\chi=0$ in a neighborhood of the zero section and $\chi=1$ near infinity. The function $\chi g_{2,\wedge}$ is then a smooth family of operators in $\L\bigl(\K^{s,-\gamma}_{\delta}, \K^{s',-\gamma+m+\varepsilon}_{\delta'}\bigr)$ for every $s,s',\delta,\delta' \in \R$ and some $\varepsilon > 0$, satisfying the estimates \eqref{GreenSymbolEstimate} with $\mu = -m$.

On each local patch near the boundary, we can build $\Op_y(\chi g_{2,\wedge})(\lambda)$ and add it to $B_2(\lambda)=\Op_y(q_2)(\lambda)$ to define a new global parameter-dependent parametrix
\begin{equation}\label{B3Parametrix}
B_3(\lambda) : x^{-\gamma}L^2_b(\M;E) \to x^{-\gamma+m}H^m_{e}(\M;E)
\end{equation}
such that $B_{3,\wedge}(y,\eta;\lambda)=B_{2,\wedge}(y,\eta;\lambda) + g_{2,\wedge}(y,\eta;\lambda)$, and
\begin{equation*}
B_3(\lambda)(A-\lambda)= I+G(\lambda): x^{-\gamma+m}H^m_{e}(\M;E)\to x^{-\gamma+m}H^m_{e}(\M;E)
\end{equation*}
with a remainder $G(\lambda)$ that, modulo terms in $\S\bigl(\Lambda,\ell^1\bigl(x^{-\gamma+m}H^m_{e}(\M;E)\bigr)\bigr)$, can locally be realized as $\Op_y(g)(\lambda)$ with $g(y,\eta;\lambda)$ in $\L\bigl(\K^{s,-\gamma+m}_{\delta},\K^{s',-\gamma+m+\varepsilon}_{\delta'}\bigr)$ satisfying the symbol estimates \eqref{GreenSymbolEstimate} with $\mu = -1$, and having an asymptotic expansion of the form
\eqref{GreenAsymptotic}.

By means of a formal Neumann series argument, one can find a family $\tilde{G}(\lambda)$ having the same structure as $G(\lambda)$ such that 
\begin{equation*}
(I+\tilde{G}(\lambda))(I+G(\lambda)) = I + R(\lambda)
\end{equation*}
with $R(\lambda) \in \S\bigl(\Lambda,\ell^1\bigl(x^{-\gamma+m}H^m_{e}(\M;E)\bigr)\bigr)$. In fact, $\tilde{G}(\lambda)$ can be chosen in such a way that $R(\lambda)$ vanishes for $\lambda \in \Lambda$ with $|\lambda|$ large enough. We then set
\begin{equation*}
B(\lambda) = (I + \tilde{G}(\lambda))B_3(\lambda)
\end{equation*}
and arrive at the following theorem.

\begin{theorem}\label{ParametrixTheorem}
The parametrix $B(\lambda) : x^{-\gamma}L^2_b \to x^{-\gamma+m}H^m_{e}$ constructed above is a left inverse of $A - \lambda : x^{-\gamma+m}H^m_{e} \to x^{-\gamma}L^2_b$ modulo compact operators. For $\lambda \in \Lambda$ with $|\lambda|$ large enough, $B(\lambda)$ is even an exact left inverse of $A - \lambda$.
\end{theorem}

Observe that $B(\lambda)$ has near the boundary the same structure as $B_1(\lambda)$, cf. \eqref{Completeedgesymbols}.  Actually, a careful inspection of the arguments in our parametrix construction above shows that, modulo negligible terms in $\S\bigl(\Lambda,\ell^1\bigl(x^{-\gamma}L^2_b, x^{-\gamma+m}H^m_{e}\bigr)\bigr)$, the difference $B(\lambda)-B_1(\lambda)$ is near the boundary of the form $\Op_y(v+g)(\lambda)$ with $v(y,\eta;\lambda)$ as in \eqref{SmoothingMellin} and $g(y,\eta;\lambda)$ satisfying \eqref{GreenSymbolEstimate} and \eqref{GreenAsymptotic} with $\mu=-m$.

Moreover, if in addition to our standing assumptions we also require that 
\begin{equation*}
A_{\wedge}(y,\eta) - \lambda : \Dom_{\wedge,\min}(y) \to x^{-\gamma}L^2_b
\end{equation*}
be surjective (and therefore bijective) on $\bigl(T^*\Y\times\Lambda\bigr)\setminus 0$, then the refinements from $B_2(\lambda)$ to $B_3(\lambda)$ to $B(\lambda)$ in the parametrix construction above furnish a parametrix $B(\lambda)$ that is also a right inverse of $A-\lambda : x^{-\gamma+m}H^m_{e} \to x^{-\gamma}L^2_b$ for $|\lambda|$ large enough. Thus, in this case, $A-\lambda$ is invertible for $\lambda \in \Lambda$ with $|\lambda|$ large, and $(A-\lambda)^{-1} =B(\lambda)$. 

\begin{remark}\label{AParametrix}
If $A \in x^{-m}\Diff^m_e(\M;E,F)$ satisfies the hypotheses of Theorem~\ref{DminTheorem}, then the above parametrix construction up to \eqref{B3Parametrix} gives (by setting $\lambda=0$) a parametrix $B$ of 
\begin{equation*}
A: x^{-\gamma+m}H^m_{e}(\M;E) \to x^{-\gamma}L^2_b(\M;F)
\end{equation*}
such that $BA-1$ is a compact operator in $\L\bigl(x^{-\gamma+m}H^m_{e}(\M;E)\bigr)$.
\end{remark} 

\section{Proof of Theorem~\ref{MinLeftInverseandExpansion}}\label{ExpansionProof}

The parametrix $B(\lambda): x^{-\gamma}L^2_b(\M;E) \to x^{-\gamma+m}H^m_{e}(\M;E)$ constructed in the previous section gives a left inverse of $A_{\min}-\lambda$; see Theorem~\ref{ParametrixTheorem}. Note that, in the situation at hand, $\Dom_{\min}(A) = x^{-\gamma+m}H^m_{e}(\M;E)$ by Theorem~\ref{DminTheorem}.

Let $\alpha=(\alpha_1,\alpha_2) \in \N_0^2$ be such that $1+|\alpha| > \frac{d}{m}$, let $\varphi \in C^{\infty}(\M;\Hom(E))$, and consider the operator family
\begin{equation*}
P(\lambda) = \varphi \partial_{\lambda}^{\alpha_1}\partial_{\overline{\lambda}}^{\alpha_2}B(\lambda): x^{-\gamma}L^2_b \to x^{-\gamma}L^2_b.
\end{equation*}
We need to prove that $P(\lambda)$ is of trace class and that $t(\lambda)=\Tr P(\lambda)$ admits the asymptotic expansion \eqref{tlambdaexpansion} as $|\lambda|\to\infty$.  Observe that for both discussions, terms in $\S\bigl(\Lambda,\ell^1(x^{-\gamma}L^2_b)\bigr)$ are negligible. Modulo such a term, $P(\lambda)$ can be written as
\begin{equation*}
P(\lambda)=P_{\textup{int}}(\lambda)+P_{\partial}(\lambda)
\end{equation*}
with $P_{\textup{int}}(\lambda)$ supported in the interior of $\M$ and $P_{\partial}(\lambda)$ supported near $\partial\M$. 

The interior term is a parameter-dependent family of pseudodifferential operators of order $-m-m|\alpha|$ on the double of $\M$ (which is a smooth compact manifold without boundary). Since $m+m|\alpha|>d=\dim\M$, we conclude that $P_{\textup{int}}(\lambda)$ is of trace class. Moreover, we have the standard asymptotic expansion
\begin{equation}\label{Pintexpansion}
\Tr P_{\textup{int}}(\lambda) \sim \sum_{j=0}^{\infty}p_j(\lambda) \text{ as } |\lambda| \to \infty,
\end{equation}
where $p_j(\varrho^m\lambda) = \varrho^{d-j-m-m|\alpha|}p_j(\lambda)$ for $\varrho > 0$.

In a coordinate patch $\Omega \times \Z^{\wedge}$ near the boundary, with coordinates $y\in\Omega$ for $\Y$, we have $P_{\partial}(\lambda) = \Op_y(q)(\lambda)$ modulo negligible terms for some operator-valued symbol $q(y,\eta;\lambda)$ taking values in the trace class operators on $\Z^\wedge$; see Remark~\ref{DerivativeRemark}. Moreover, as a symbol with values in $\ell^1\bigl(x^{-\gamma}L^2_b\bigr)$, the order of $q(y,\eta;\lambda)$ is less than $-\dim\Y$, so $P_{\partial}(\lambda)$ is of trace class with
\begin{equation}\label{TraceDarstellung}
\Tr P_{\partial}(\lambda) = \frac{1}{(2\pi)^q}\iint \tr_{\Z^{\wedge}}\bigl(q(y,\eta;\lambda)\bigr)\,dy\,d \eta,
\end{equation}
where $\tr_{\Z^{\wedge}}$ is the trace functional on $\ell^1\bigl(x^{-\gamma}L^2_b\bigr)$ on $\Z^{\wedge}$.

Now, for every $y\in\Omega$, the symbol $q(y,\eta;\lambda)$ is a family of cone operators on $\Z^\wedge$ with parameters $(\eta,\lambda)$; isotropic in $\eta\in\R^q$ and anisotropic in $\lambda\in\Lambda$. By adapting the results from \cite[Section 4]{GKM4} to our present situation, we obtain the expansion
\begin{equation}\label{FiberExpansion}
\tr_{\Z^{\wedge}}\bigl(q(y,\eta;\lambda)\bigr) \sim \sum_{j=0}^{\infty}\sum_{k=0}^{m_j} c_{jk}(y,\eta,\lambda)\log^k(|\eta,\lambda|_m) \text{ as } |\eta,\lambda|_m \to \infty,
\end{equation}
where the $c_{jk}$ satisfy the homogeneity relations
$$
c_{jk}(y,\varrho\eta,\varrho^m\lambda) = \varrho^{n+1-j-m-m|\alpha|}c_{jk}(y,\eta,\lambda) \text{ for } \varrho > 0.
$$
Here $m_j \leq 1$ for all $j$, and $m_j = 0$ for $j \leq n=\dim \Z$. 

In other words, $\tr_{\Z^{\wedge}}\bigl(q(y,\eta;\lambda)\bigr)$ is a parameter-dependent $\log$-polyhomogeneous symbol; see \cite{Le99}. Making use of the asymptotic expansion \eqref{FiberExpansion}, the integral \eqref{TraceDarstellung}, and the typical arguments involving change of variables, homogeneity, and the multiplicative properties of the logarithm, we then get the expansion
\begin{equation}\label{Pdexpansion}
\Tr P_{\partial}(\lambda) \sim
\sum_{j=0}^{\infty}\sum_{k=0}^{m_j}d_{jk}(\lambda)\log^k(|\lambda|) \text{ as } |\lambda| \to \infty,
\end{equation}
where $d_{jk}(\varrho^m\lambda) = \varrho^{d-j-m-m|\alpha|}d_{jk}(\lambda)$ for $\varrho > 0$. As before, $m_j \leq 1$ for all $j$, and $m_j = 0$ for $j \leq n$.

Both expansions \eqref{Pintexpansion} and \eqref{Pdexpansion} imply the asymptotic expansion \eqref{tlambdaexpansion}, and the proof of Theorem~\ref{MinLeftInverseandExpansion} is complete. \qed


\end{document}